%% file: paper.tex
\documentclass[onefignum,onetabnum]{siamart190516}
\renewcommand{\theequation}{\thesection.\arabic{equation}}
\renewcommand{\thefigure}{\thesection.\arabic{figure}}
\renewcommand{\thetable}{\thesection.\arabic{table}}

\usepackage{amsmath} 
\usepackage{amsfonts}
\usepackage{amssymb}
\usepackage{booktabs}
\usepackage{stmaryrd}
\usepackage{subcaption}
\usepackage{cite}
\usepackage{cleveref}

\input{ex_shared}

\newcommand{\Ltwo}{L^2}

\newcommand{\hf}{\frac{1}{2}}

\newcommand{\TVM}{\mathrm{TVM}}
\newcommand{\cF}{\mathcal{F}}
\newcommand{\cG}{\mathcal{G}}

\numberwithin{equation}{section}

\usepackage[belowskip=-3pt]{caption}
\begin{document}

\title{TVD and TVB Preservation without TVD Time Discretization for Discontinuous Galerkin Methods}

\maketitle

\begin{abstract}
Total variation diminishing (TVD) and total variation bounded (TVB) properties are crucial for controlling spurious oscillations in numerical solutions of conservation laws. In the classical Runge--Kutta (RK) discontinuous Galerkin (DG) framework, enforcing these properties is intrinsically tied to TVD time integrators, more commonly known today as strong-stability-preserving (SSP) methods. This reliance imposes severe structural restrictions, including order barriers and incompatibility with various fully discrete DG formulations, ranging from the recent RKDG method with compact stencils (cRKDG) to the widely established Arbitrary DERivative (ADER) DG method. To bypass these constraints, we propose a novel trace-limited corrector framework that preserves the TVD/TVB-in-the-means properties using generic, non-SSP time stepping. Based on Harten's lemma, our key insight is that total variation stability is dictated solely by the cell-average update in the final corrector stage. Consequently, we modify the traces in the numerical fluxes exclusively in the final stage, leaving the intermediate predictor stages unconstrained. This strategy decouples oscillation control from the SSP restriction, accommodates standard RKDG, cRKDG, and ADER-DG predictors, and retains the compactness of the cRKDG framework. We also prove that the limiter does not activate in smooth regions, thereby preserving the underlying accuracy. Finally, numerical experiments are presented to demonstrate the capabilities and robustness of the method.
\end{abstract} 

\begin{keywords}
discontinuous Galerkin methods, hyperbolic conservation laws, total variation diminishing, total variation bounded, Harten's lemma, non-SSP time discretizations
\end{keywords}

\begin{AMS}
65M60, 65M12
\end{AMS}

\vspace{-0.2cm}

\section{Introduction}
\label{sec:introduction}

Discontinuous Galerkin (DG) methods, typically coupled with Runge--Kutta (RK) time discretizations, are widely used for solving hyperbolic conservation laws \cite{rkdg1,rkdg2,rkdg3,rkdg4,rkdg5}. To suppress spurious oscillations near shocks generated by the high-order approximations, total variation diminishing (TVD) or total variation bounded (TVB) limiting is usually applied as a standard remedy \cite{rkdg2}. In the classical Runge--Kutta discontinuous Galerkin (RKDG) framework,  this limiting strategy is intrinsically tied to what were historically termed TVD RK time discretizations \cite{shu1988total,gottlieb1998total}, more commonly known today as strong-stability-preserving (SSP) RK time discretizations \cite{gottlieb2001strong,gottlieb2011strong}. These methods can be represented as convex combinations of forward-Euler steps. 
By ensuring that the forward-Euler step is TVD/TVB in the means, the SSP-RK method inherits this stability property via its underlying convex decomposition.  Consequently, the TVD or SSP time integrators were historically believed to be essential for preserving fully discrete TVD properties; indeed, numerical examples in \cite{gottlieb1998total} demonstrated that ``non-TVD but linearly stable RK time discretization can generate oscillations even for TVD spatial discretization."

To overcome this rigid reliance on SSP structures, we develop a novel TVD/TVB framework tailored for general non-SSP time stepping. Our motivation is twofold. First, explicit SSP-RK methods impose strict order barriers: they require additional stages to achieve fourth-order accuracy and cannot exceed order four with positive weights \cite{gottlieb1998total, ruuth2002two}. Second, many advanced fully discrete DG formulations naturally lack an SSP structure; for instance, the RKDG method with compact stencils (cRKDG) fundamentally requires non-SSP-RK schemes to ensure conservation \cite{chen2024runge, liu2025bound, babbar2025compact, babbar2026compact}, while the Arbitrary DERivative (ADER) DG \cite{dumbser2005ader,dumbser2008unified,dumbser2018efficient,gaburro2023high} method
does not use RK or multi-step integrators at all, let alone their SSP versions.

Our approach relies on the same theoretical foundation as the classical SSP-RKDG 
framework, namely Harten's lemma 
\cite{shu2009discontinuous,harten1983high}, but applies it from a different perspective. Drawing on this lemma, our key observation is that the TVD/TVB-in-the-means property is dictated solely by the cell-average update in the final corrector stage of the RK integrator. Consequently, 
this property can be attained by modifying the numerical fluxes in the final 
stage alone. The interior predictor stages are not constrained by this property; 
they may use generic {approximations} without an SSP structure. We leverage this insight 
by introducing a trace-limited corrector framework for DG schemes paired with 
generic {non-SSP} time discretizations. The predictor stages are left unchanged, while their traces entering the final numerical fluxes in the corrector stage are limited with respect to the initial-stage cell averages, thereby ensuring that the final-stage DG solution is TVD/TVB in the means.
A standard polynomial limiter is then applied as a postprocessing step to produce the accepted DG polynomial beyond its means.
This construction decouples the TVD/TVB limiting mechanism from the SSP structure of the time integrator, while preserving conservation and the
usual TVD/TVB-in-the-means interpretation of DG limiting.

As a main application of the proposed framework, we construct the TVD/TVB cRKDG method. The cRKDG method \cite{chen2024runge} utilizes cell-local derivative operators during the inner stages of an RK method, yielding compact stencils, reduced communication overhead in parallel computing, and simple boundary treatments. However, due to its fundamental reliance on a non-SSP-RK structure, no TVD/TVB cRKDG schemes have been designed in the literature prior to this work. Beyond cRKDG, the proposed framework is also compatible with other fully discrete DG formulations; for instance, Section \ref{sec:ADER} outlines how it can be combined with ADER-DG methods to achieve provable TVD/TVB properties. Moreover, the new framework allows implicit predictors, thereby opening new pathways for designing implicit TVD/TVB DG schemes. {To keep our work more focused, we will concentrate on explicit RKDG and cRKDG solvers in subsequent sections; detailed studies on extensions to other fully discrete DG methods will be postponed to future work.}

This work aligns with active ongoing efforts in the literature to address two central challenges in numerical conservation laws: oscillation control \cite{zhu2020simple,lu2021oscillation,peng2025oedg,wei2025jump} and structure preservation \cite{guermond2019invariant,ranocha2020relaxation,kuzmin2024property,rueda2024monolithic,wu2026high}. In particular, we note that the proposed method shares similar goals and outcomes with recent invariant-domain-preserving RK methods (see, e.g., \cite{ern2022invariant}), as both preserve physical invariant properties beyond the SSP structural restriction. However, our approach adopts a fundamentally different methodology. 
By building upon the classical Harten's lemma, this approach eliminates the need for immediate limiting operations following each intermediate predictor stage, thereby retaining the compactness of the cRKDG method, streamlining the algorithmic framework for concurrent computation, and facilitating better adaptation to other fully discrete DG formulations beyond the RK family.

The remainder of the paper is organized as follows. Section~\ref{sec:ssp-rkdg} reviews the classical SSP-RKDG framework and the TVD/TVB polynomial limiter. Section~\ref{sec:generic-rkdg} introduces the proposed trace-limited corrector framework for generic RKDG and cRKDG methods, {along with a brief discussion on its application to the ADER-DG method}. The main properties of the proposed schemes are analyzed in Section~\ref{sec:properties}.
Extensions to systems and multidimensional problems are described in
Section~\ref{sec:extensions}. Numerical tests are presented in
Section~\ref{sec:numerical_tests}. Conclusions are given in Section~\ref{sec:conclusions}. 

\section{Background}
\label{sec:ssp-rkdg}
This section reviews the classical SSP-RKDG framework, in which
TVD or TVB stability is inherited from the limited forward-Euler DG step.

\subsection{Forward-Euler DG methods}
\label{subsec:dg-fe}
Consider a scalar conservation law in 1D with periodic or compactly supported boundary conditions:
\begin{equation}
  u_t + f(u)_x = 0, \qquad x\in \Omega, \quad t>0.
  \label{eq:scl}
\end{equation}

Let $\mathcal{T}_h = \{I_i\}_{i=1}^N$, with $I_i=[x_{i-1/2},x_{i+1/2}]$,
be a partition of \(\Omega\). We denote by
\(h_i=x_{i+1/2}-x_{i-1/2}\) the cell size of $I_i$ and $h = \max_i h_i$. The DG finite element space of degree \(k\) is $\mathbb{V}_h^k
  =
  \bigl\{v\in L^2(\Omega):
  v|_{I_i}\in \mathbb P^k(I_i),\; i=1,\ldots,N\bigr\},$
where $\mathbb{P}^k(I_i)$ denotes the space of polynomials of degree at most $k$ on $I_i$. For \(v\in \mathbb{V}_h^k\), we denote its traces at cell interfaces by
$v_{i+1/2}^\pm = \lim_{\delta\to 0^{+}} v(x_{i+1/2}\pm\delta)$, 
and write $\overline v_{i}
  =
  {1}/{h_i}\cdot\int_{I_i}v(x)\,\mathrm{d}x$ 
for its cell average on \(I_i\). 
We also use the shorthands
$\Delta_+\overline v_{i} = \overline v_{i+1}-\overline v_{i}$ and $
  \Delta_-\overline v_{i} = \overline v_{i}-\overline v_{i-1}$
for the forward and backward differences of cell averages, respectively.

The standard DG semi-discretization is to find \(u_h(\cdot,t)\in \mathbb{V}_h^k\)
such that
\begin{equation}
  (u_h)_t = \mathcal F(u_h),
  \label{eq:semi-dg}
\end{equation}
where the DG spatial operator \(\mathcal F:\mathbb{V}_h^k\to \mathbb{V}_h^k\) is defined by
\begin{equation}
\begin{aligned}
  \int_{I_i}\mathcal F(u_h)v\,\mathrm{d}x
  =
  \int_{I_i} f(u_h)v_x\,\mathrm{d}x -
  \left[
  \widehat f(u_{h,i+\hf}^-,u_{h,i+\hf}^+)v_{i+\hf}^-
  -
  \widehat f(u_{h,i-\hf}^-,u_{h,i-\hf}^+)v_{i-\hf}^+
  \right]
\end{aligned}
\label{eq:F-def}
\end{equation}
for all \(v\in \mathbb{V}_h^k\). Here, \(\widehat f \) is a consistent monotone
numerical flux, which means \(\widehat f(u,u)=f(u)\), and \(\widehat f\) is nondecreasing in its first argument and
nonincreasing in its second argument. 
Moreover, we assume that
\(\widehat f\) is Lipschitz continuous in both arguments.

The raw forward-Euler DG step associated with \eqref{eq:semi-dg} is
\begin{equation}
  u_h^{n+1}
  =
  u_h^n+\Delta t\,\mathcal F(u_h^n).
  \label{eq:fe-dg}
\end{equation}
Taking
\(v=1\) in \eqref{eq:F-def} gives the cell-average update
\begin{equation}
  \overline u_{h,i}^{n+1}
  =
  \overline u_{h,i}^n
  -\lambda_i
  \left[
  \widehat f(u_{h,i+1/2}^{n,-},u_{h,i+1/2}^{n,+})
  -
  \widehat f(u_{h,i-1/2}^{n,-},u_{h,i-1/2}^{n,+})
  \right],
  \quad
  \lambda_i=\frac{\Delta t}{h_i}.
  \label{eq:fe-cell-average}
\end{equation}
Thus the cell averages evolve by a conservative finite-volume-type update, but with numerical fluxes evaluated at DG traces.

A relevant stability notion is defined through the total variation of cell averages. 
\begin{definition}
Let $\operatorname{TVM}(v) = \sum_i |\overline{v}_{i+1} - \overline{v}_{i}|$ denote the total-variation-in-the-means seminorm on $\mathbb{V}_h^k$. A DG method is:
\begin{enumerate}
    \item {TVD in the means} (TVDM) if $\operatorname{TVM}(u_h^{n+1}) \le \operatorname{TVM}(u_h^n)$;
    \item {TVB in the means} (TVBM) if $\operatorname{TVM}(u_h^{n}) \leq C \ \forall n \ge 0$, where $C$ is a constant.
\end{enumerate}
\end{definition}
  
The following lemma provides a sufficient condition for TVDM.

\begin{lemma}[Harten's lemma \cite{harten1983high}]
\label{lem:harten}
If a cell-average scheme admits the form
\begin{equation}
  \overline u_{h,i}^{n+1}
  =
  \overline u_{h,i}^n
  + C_{i+1/2}\Delta_+\overline u_{h,i}^n
  - D_{i-1/2}\Delta_-\overline u_{h,i}^n,
  \label{eq:harten-form}
\end{equation}
with periodic or compactly supported boundary conditions, where $C_{i+1/2}$ and $D_{i-1/2}$ may be nonlinear functions of $u_h^n$ satisfying 
\begin{equation}
  C_{i+1/2}\ge 0,
  \qquad
  D_{i+1/2}\ge 0,
  \qquad
  C_{i+1/2}+D_{i+1/2}\le 1
  \label{eq:harten-condition}
\end{equation}
for all \(i\), then the scheme is TVDM:
\begin{equation}
  \operatorname{TVM}(u_h^{n+1})
  \le
  \operatorname{TVM}(u_h^n).
  \label{eq:limited-fe-tvdm}
\end{equation}
\end{lemma}

\subsection{TVD/TVB limiting}
\label{subsec:poly-limiter}

The raw forward-Euler update \eqref{eq:fe-dg} is not automatically TVDM/TVBM for
general DG polynomials, as the interface traces may differ from the
cell averages by amounts that violate the Harten form. 
The classical remedy is to apply a limiter to process the candidate DG solution \cite{rkdg2,shu2009discontinuous, zhong2013simple, dumbser2014posteriori}.

The standard minmod function is
\begin{equation}
  m(a,b,c)
  =
  \begin{cases}
  s\min\{|a|,|b|,|c|\},
  &\text{if } \operatorname{sgn}(a)=\operatorname{sgn}(b)=\operatorname{sgn}(c)=s,\\[2mm]
  0,
  &\text{otherwise}.
  \end{cases}
  \label{eq:minmod}
\end{equation}
The modified minmod function introduces a parameter \(M\ge 0\) that relaxes
\eqref{eq:minmod} by
\begin{equation}
  \widetilde m(a,b,c;M,h)
  =
  \begin{cases}
  a,
  &\text{if } |a|\le Mh^2,\\[1mm]
  m(a,b,c),
  &\text{otherwise}.
  \end{cases}
  \label{eq:tvb-minmod}
\end{equation}
The choice \(M=0\) gives the TVD limiting, while \(M>0\) gives the TVB
limiting.

We now define the limiter $\Lambda:\mathbb{V}_h^k\to \mathbb{V}_h^k$. Let \(w\in \mathbb{V}_h^k\) be a candidate DG function. On cell \(I_i\), define the limited traces by
\begin{subequations}
\label{eq:limited-traces-for-poly}
\begin{align}
  \widetilde w_{i+1/2}^-
  &=
  \overline w_{i}
  +
  \widetilde m
  \left(
  w_{i+1/2}^- - \overline w_{i},
  \Delta_+\overline w_{i},
  \Delta_-\overline w_{i};
  M,h_i
  \right),
  \label{eq:limited-right-trace-poly}
  \\
  \widetilde w_{i-1/2}^+
  &=
  \overline w_{i}
  -
  \widetilde m
  \left(
  \overline w_{i} - w_{i-1/2}^+,
  \Delta_+\overline w_{i},
  \Delta_-\overline w_{i};
  M,h_i
  \right).
  \label{eq:limited-left-trace-poly}
\end{align}
\end{subequations}
A cell \(I_i\) is called 
\emph{troubled} if at least one of the two evaluations of \(\widetilde m\) in
\eqref{eq:limited-traces-for-poly} modifies its first argument. 
On cells that are not troubled, the limiter leaves the polynomial unchanged, i.e., $\Lambda(w)|_{I_i}=w|_{I_i}.$ 
On a troubled cell \(I_i\), the limiter replaces \(w|_{I_i}\) by the unique reconstructed polynomial in \(\mathbb P^{\min(k,2)}(I_i)\) that preserves the cell average while agreeing with the limited trace values in \eqref{eq:limited-traces-for-poly}, i.e.,
\begin{equation}
\overline {\Lambda(w)}_{i} = \overline{w}_i, \qquad 
  \bigl(\Lambda(w)\bigr)_{i+1/2}^{-}
  =
  \widetilde w_{i+1/2}^{-},
  \qquad
  \bigl(\Lambda(w)\bigr)_{i-1/2}^{+}
  =
  \widetilde w_{i-1/2}^{+}.
\end{equation}

Using the convention above, the limited forward-Euler DG method is
written as
\begin{equation}
  w_h^{n+1}
  =
  u_h^n+\Delta t\,\mathcal F(u_h^n),
  \qquad
  u_h^{n+1}
  =
  \Lambda(w_h^{n+1}).
  \label{eq:limited-fe-dg}
\end{equation}
Here \(u_h^n\) is the accepted DG solution at time \(t^n\), while
\(w_h^{n+1}\) is the unlimited candidate produced by the forward-Euler step.
Since \(\Lambda\) preserves cell averages, the cell-average update of
\(u_h^{n+1}\) is the same as that of \(w_h^{n+1}\).

The role of the limiter is to ensure that the traces of the accepted input
\(u_h^n\) are compatible with the neighboring cell averages. For the TVD
case \(M=0\), under the standard CFL condition associated with the monotone
flux, the cell-average update can be written in Harten form
\eqref{eq:harten-form}--\eqref{eq:harten-condition}. Consequently, the solution to \eqref{eq:limited-fe-dg} satisfies \eqref{eq:limited-fe-tvdm} and is TVDM. 
\begin{remark}[{TVBM}]\label{rem:tvbm}
    With TVB limiting ($M > 0$), the TVM seminorm may grow since $\widetilde{m}$ may return the argument $a$, which need not satisfy TVD bounds when $|a| \le Mh^2$. 
    However, the excess in each cell is bounded by $O(h^2)$, leading to a total variation growth of $O(h)$ per time step and hence remaining TVBM over any finite time interval \cite{shu1987tvb}. The TVB modification is crucial for retaining high-order accuracy near smooth extrema, where a strict TVD limiter would otherwise flatten it.
\end{remark}

\subsection{TVD/TVB stability in SSP-RKDG methods}
\label{subsec:ssp-rkdg}
SSP-RK methods are designed as convex combinations of forward-Euler steps. 
Therefore, if the forward-Euler building block is stable in certain seminorms under a time-step restriction, then the SSP method inherits the same stability under the corresponding CFL
restriction.

In the RKDG setting, this principle is combined with the limited
forward-Euler step \eqref{eq:limited-fe-dg} for the TVD/TVB stability. For example, the commonly used
third-order SSP-RKDG method can be written in the following form:
\begin{equation}
\begin{aligned}
  w_h^{(1)}
  &=
  u_h^n+\Delta t\,\mathcal F(u_h^n),
  &
  u_h^{(1)}
  &=
  \Lambda(w_h^{(1)}),
  \\
  w_h^{(2)}
  &=
  u_h^{(1)}+\Delta t\,\mathcal F(u_h^{(1)}),
  &
  u_h^{(2)}
  &=
  \frac34 u_h^n+
  \frac14\Lambda(w_h^{(2)}),
  \\
  w_h^{(3)}
  &=
  u_h^{(2)}+\Delta t\,\mathcal F(u_h^{(2)}),
  &
  u_h^{n+1}
  &=
  \frac13 u_h^n+
  \frac23\Lambda(w_h^{(3)}).
\end{aligned}
\label{eq:ssp-rk3-dg}
\end{equation}
Each \(w_h^{(j)}\) is an unlimited forward-Euler candidate, and each
\(\Lambda(w_h^{(j)})\) is the corresponding limited DG solution. The
accepted stage values are convex combinations of limited forward-Euler
outputs and previously accepted data.

For \(M=0\), the limited forward-Euler step is TVDM under the
appropriate CFL condition. Since \(\operatorname{TVM}\) is a seminorm and \eqref{eq:ssp-rk3-dg} is a convex combination of forward-Euler steps, the
TVDM inequality \eqref{eq:limited-fe-tvdm} is inherited by the SSP-RKDG method. With the TVB limiter \(M>0\), the same convexity argument gives its TVBM version.

The essential point is that the classical proof uses the SSP decomposition
of the time integrator. TVD/TVB stability is obtained because every stage can
be interpreted as a convex combination of limited forward-Euler updates. This
is a powerful framework, but it also ties the limiting strategy to the SSP
structure of the RK method. Generic RK methods that do
not admit such a convex forward-Euler representation are therefore outside
this classical argument.

\begin{remark}[{SSP restrictions}]
The SSP requirement imposes well-known \emph{order barriers} \cite{gottlieb1998total, ruuth2002two}. 
In particular, the classical four-stage fourth-order RK method is not an SSP method in the sense needed by the above TVD argument. Thus, within the
traditional SSP-RKDG framework, the choice of temporal discretization is
restricted by the availability of an SSP representation rather than by formal
accuracy alone.
\end{remark}

\section{New Framework for non-SSP Methods}
\label{sec:generic-rkdg}

We now introduce a framework distinct from the SSP argument in Section~\ref{sec:ssp-rkdg}, {which treats a fully discrete scheme} as a sequence of predictors at intermediate stages followed by a corrector at the final stage. 
The TVD/TVB structure is then imposed solely by modifying the numerical fluxes in the final corrector. This approach is independent of the predictors, and hence it applies 
universally to the standard RKDG, cRKDG, and ADER-DG methods. In the standard RKDG 
case, the predictors are computed via RK approximations using the usual 
DG operator $\mathcal{F}$; in the cRKDG case, the predictors are computed 
via RK approximations using a compact cell-local operator $\mathcal{G}$; in the 
ADER-DG case, the predictors are computed via a local space-time 
Galerkin method.

To maintain focus, we concentrate on RK formulations, 
discussing RKDG and cRKDG schemes in detail, while only briefly commenting 
on the application to the ADER-DG method in 
Section~\ref{sec:ADER}.

\subsection{Generic predictor-corrector formulations}
\label{subsec:generic-predictors}

Consider an explicit \(s\)-stage RK method with the Butcher tableau
\begin{equation}\label{eq:Butcher-RK}
    \begin{array}{c|c}
        c & A \\
        \hline
          & b 
    \end{array}
    \quad c = (c_1,\cdots,c_s)^\top, 
    \  b = (b_1,\cdots,b_s),
    \  A = (a_{\ell j})_{s\times s}, 
    \  a_{\ell j}=0 \text{ if } j \ge \ell.
\end{equation}

The standard RKDG method applies this RK integrator to the semi-discrete DG system
\eqref{eq:semi-dg}. Starting from an accepted solution \(u_h^n\), the
predictor stages are
\begin{equation}
  u_h^{(\ell)}
  =
  u_h^n
  +\Delta t\sum_{j=1}^{\ell-1}a_{\ell j}\mathcal F(u_h^{(j)}),
  \qquad \ell=1,\ldots,s,
  \label{eq:rkdg-predictor}
\end{equation}
and the corresponding raw RKDG corrector is
\begin{equation}
  u_h^{n+1}
  =
  u_h^n
  +\Delta t\sum_{\ell=1}^s b_\ell\mathcal F(u_h^{(\ell)}).
  \label{eq:rkdg-corrector-raw}
\end{equation}
If the RK method is not SSP, then the  argument in
Section~\ref{subsec:ssp-rkdg} does not provide a TVDM/TVBM mechanism for
\eqref{eq:rkdg-predictor}--\eqref{eq:rkdg-corrector-raw}.

For the cRKDG method, we define the cell-local operator
\(\mathcal G:\mathbb{V}_h^k\to \mathbb{V}_h^k\) by
\begin{equation}
  \int_{I_i}\mathcal G(u_h)v\,\mathrm{d}x
  =
  \int_{I_i}f(u_h)v_x\,\mathrm{d}x -
  \left[
  f(u_{h,i+1/2}^{-})v_{i+1/2}^{-}
  -
  f(u_{h,i-1/2}^{+})v_{i-1/2}^{+}
  \right]
\label{eq:G-def}
\end{equation}
for all \(v\in \mathbb{V}_h^k\). Compared with \(\mathcal F\), the operator \(\mathcal G\) uses only the inner traces from the current cell and therefore acts locally on each cell.
The cRKDG method replaces \(\mathcal F\) by \(\mathcal G\) in the predictor stages \eqref{eq:rkdg-predictor} and keeps \(\mathcal F\) in the corrector \eqref{eq:rkdg-corrector-raw}.

\begin{remark}[Incompatibility between cRKDG and SSP-RK]
Note that cRKDG schemes cannot be constructed with SSP-RK methods, otherwise the nonconservative operator $\cG$ would then be used to update $u_h^{n+1}$, violating conservation. Hence, to construct a TVD/TVB cRKDG method, the approach in Section~\ref{sec:ssp-rkdg} does not work. However, its TVD/TVB properties can be preserved within the new framework. 
\end{remark}

It is helpful to write both RKDG and cRKDG predictors in the unified form
\begin{equation}\label{eq:generic-predictor}
  u_h^{(\ell)}
  =
  u_h^n
  +\Delta t\sum_{j=1}^{\ell-1}a_{\ell j}\mathcal H(u_h^{(j)}),
  \qquad
  \ell=1,\ldots,s,
  \qquad
  \mathcal H\in\{\mathcal F,\mathcal G\}.
\end{equation}
The trace-limited corrector to be introduced is the same for both choices.

\subsection{Trace limiting and modified corrector operator}
\label{subsec:trace-limiter}

The polynomial limiter \(\Lambda\) defined in Section~\ref{subsec:poly-limiter}
returns a DG polynomial and preserves cell averages. The new limiter $\Lambda_{\mathrm{tr}}$ introduced
here has a different purpose. It modifies only the traces used in the final
numerical flux and does not reconstruct a DG polynomial.

We denote by $\mathbb{T}_h^k = \{ v|_{\partial \mathcal{T}_h} : v \in \mathbb{V}_h^k \}$ the space of double-valued traces for functions in $\mathbb{V}_h^k$, where $\partial \mathcal{T}_h = \bigcup_{i=1}^N \{ \partial I_i \}$ represents the collection of all element boundaries. 
Let $v \in \mathbb{V}_h^k$ be a DG function providing the raw traces to be modified, and let $\bar{q} \in \mathbb{V}_h^0$ be a piecewise constant function prescribing the reference cell averages. 
The trace limiter is defined as the mapping $\Lambda_{\mathrm{tr}}: \mathbb{V}_h^k \times \mathbb{V}_h^0 \to \mathbb{T}_h^k$, where the output $\widetilde{v} = \Lambda_{\mathrm{tr}}(v, \bar{q}) = \{ \widetilde{v}_{i-1/2}^+, \widetilde{v}_{i+1/2}^- \}_{i=1}^N$
denotes the limited traces on $\partial \mathcal{T}_h$ such that
\begin{subequations}
\label{eq:trace-limiter-def}
\begin{align}
  \widetilde v_{i+1/2}^{-}
  &=
  \overline q_i
  +
  \widetilde m
  \left(
  v_{i+1/2}^{-}-\overline q_i,
  \Delta_+\overline q_i,
  \Delta_-\overline q_i;
  M,h_i
  \right),
  \label{eq:trace-limiter-right}
  \\
  \widetilde v_{i-1/2}^{+}
  &=
  \overline q_i
  -
  \widetilde m
  \left(
  \overline q_i-v_{i-1/2}^{+},
  \Delta_+\overline q_i,
  \Delta_-\overline q_i;
  M,h_i
  \right).
  \label{eq:trace-limiter-left}
\end{align}
\end{subequations}
Note that \eqref{eq:trace-limiter-def} generalizes \eqref{eq:limited-traces-for-poly} by allowing $\overline{q}_i \neq \overline{v}_{i}$. The raw traces and reference cell averages can be completely unrelated.

We now generalize the DG operator in \eqref{eq:F-def} by separating the volume and facial contributions, which yields $\cF(\cdot;\cdot):\mathbb{V}_h^k\times \mathbb{T}_h^k\to \mathbb{V}_h^k$. Given \(v\in \mathbb{V}_h^k\) and limited traces
\(\widetilde v \in \mathbb{T}_h^k\), we define \(\mathcal F(v;\widetilde v)\in \mathbb{V}_h^k\) by setting
\begin{equation}
\begin{aligned}
  \int_{I_i}\mathcal F(v;\widetilde v)\phi\,\mathrm{d}x
  =\int_{I_i}f(v)\phi_x\,\mathrm{d}x -
  \left[
  \widehat f(\widetilde v_{i+\hf}^{-},\widetilde v_{i+\hf}^{+})
  \phi_{i+\hf}^{-}
  -
  \widehat f(\widetilde v_{i-\hf}^{-},\widetilde v_{i-\hf}^{+})
  \phi_{i-\hf}^{+}
  \right]
\end{aligned}
\label{eq:F-two-argument}
\end{equation}
for all \(\phi\in \mathbb{V}_h^k\). 
The volume term is evaluated using the input raw polynomial \(v\), while only
the numerical fluxes are evaluated using the limited traces \(\widetilde v\).
If the limited traces coincide with the
original traces of \(v\), then \(\mathcal F(v;\widetilde v)=\mathcal F(v)\).

\subsection{Proposed trace-limited scheme}
\label{subsec:proposed-scheme}

We now define the proposed method. Let $u_h^n\in \mathbb{V}_h^k$ be the accepted DG solution at time $t^n$, and let $\overline{u}_h^n \in \mathbb{V}_h^0$ be the piecewise constant function defined by the cell averages. First, we compute the predictor stages via \eqref{eq:generic-predictor}, with either the RKDG or cRKDG predictors.
Next, apply the trace limiter to each predictor stage using the initial cell averages $\overline{u}_h^n$ as the reference:
\begin{equation}
  \widetilde{u}_h^{(\ell)} = \Lambda_{\mathrm{tr}} \left( u_h^{(\ell)}; \overline{u}_h^n \right), \qquad \ell=1,\ldots,s.
  \label{eq:proposed-trace-limiting}
\end{equation}
Then, the corrected candidate at the new time level is defined by
\begin{equation}
  w_h^{n+1}
  =
  u_h^n
  +\Delta t\sum_{\ell=1}^s
  b_\ell\mathcal F(u_h^{(\ell)};\widetilde u_h^{(\ell)}).
  \label{eq:proposed-corrector}
\end{equation}
Finally, the accepted solution is obtained by the polynomial limiter
\begin{equation}
  u_h^{n+1}
  =
  \Lambda(w_h^{n+1}).
  \label{eq:proposed-post-limiter}
\end{equation}
Equations \eqref{eq:generic-predictor}, \eqref{eq:proposed-trace-limiting}--\eqref{eq:proposed-post-limiter}
form the proposed trace-limited RKDG/cRKDG scheme.

The proposed scheme separates the predictors from the limiting
mechanism. In \eqref{eq:generic-predictor}, the choice \(\mathcal H=\mathcal F\)  gives a trace-limited RKDG method, while \(\mathcal H=\mathcal G\) gives a trace-limited cRKDG method.
The final corrector \eqref{eq:proposed-corrector} and the final polynomial
limiter \eqref{eq:proposed-post-limiter} are identical in both cases.

\begin{remark}[{Reference cell averages}]
The crucial point is that the reference averages in
\eqref{eq:proposed-trace-limiting} are the old-time averages
\(\overline{u}_h^n\), rather than the predictor-stage averages
\(\overline{u}_h^{(\ell)}\). This choice is what allows the final
cell-average corrector \eqref{eq:proposed-corrector} to be analyzed through
Harten's lemma. The predictor stages themselves are not required to satisfy an
SSP decomposition, nor are they limited as DG polynomials. Note that the trace limiters \eqref{eq:proposed-trace-limiting} can be evaluated in parallel without interfering with the sequential predictor updates, potentially yielding better parallel efficiency.
\end{remark}

Throughout the analysis below, we assume that the corrector weights satisfy
\begin{equation}
  b_\ell\ge 0,
  \qquad
  \sum_{\ell=1}^s b_\ell=1.
  \label{eq:positive-b-weights}
\end{equation}
This condition is much weaker than requiring the full RK method to 
be SSP. It is satisfied by many commonly used RK integrators, such as 
Heun’s third-order method and the classical fourth-order RK method discussed 
in the following subsection. However, some specially constructed RK time 
integrators may not satisfy this condition, such as the counterexample 
in \cite{gottlieb1998total} for demonstrating the necessity of SSP 
time integrators.

\subsection{Concrete cRKDG schemes}
\label{subsec:concrete-crkdg}

We now list concrete cRKDG schemes used in the numerical tests. The RKDG versions are obtained by replacing $\mathcal{G}$ by $\mathcal{F}$. 

\emph{Heun’s third-order method:}
\begin{equation}
\begin{aligned}
  u_h^{(1)}&=u^{n}_{h},\;\; u_h^{(2)}
  =
  u_h^n+\frac{\Delta t}{3}\mathcal G(u_h^{(1)}),\;\; 
  u_h^{(3)}
  =
  u_h^n+\frac{2\Delta t}{3}\mathcal G(u_h^{(2)}),
  \\
  \widetilde u_h^{(\ell)}
  &=
  \Lambda_{\mathrm{tr}}
  \left(u_h^{(\ell)};\overline{u}_h^n\right),\quad \ell=1,3,
  \\
  w_h^{n+1}
  &=
  u_h^n
  +\Delta t
  \left[
  \frac14\mathcal F(u_h^{(1)};\widetilde u_h^{(1)})
  +
  \frac34\mathcal F(u_h^{(3)};\widetilde u_h^{(3)})
  \right],
  \\
  u_h^{n+1}
  &=
  \Lambda(w_h^{n+1}).
\end{aligned}
\label{eq:crkdg3-limited}
\end{equation}

\emph{Classic fourth-order method:}
\begin{equation}
\begin{aligned}
  u_h^{(1)}&=u_{h}^{n},\; u_h^{(2)}
  =
  u_h^n+\frac{\Delta t}{2}\mathcal G(u_h^{(1)}), \; 
  u_h^{(3)}
  =
  u_h^n+\frac{\Delta t}{2}\mathcal G(u_h^{(2)}),
  \;
  u_h^{(4)}
  =
  u_h^n+\Delta t\,\mathcal G(u_h^{(3)}),
  \\
  \widetilde u_h^{(j)}
  &=
  \Lambda_{\mathrm{tr}}
  \left(u_h^{(j)};\overline{u}_h^n\right),
  \qquad j=1,2,3,4,
  \\
  w_h^{n+1}
  &=
  u_h^n
  +\Delta t
  \left[
  \frac16\mathcal F(u_h^{(1)};\widetilde u_h^{(1)})
  +
  \frac13\mathcal F(u_h^{(2)};\widetilde u_h^{(2)})
  +
  \frac13\mathcal F(u_h^{(3)};\widetilde u_h^{(3)})
  +
  \frac16\mathcal F(u_h^{(4)};\widetilde u_h^{(4)})
  \right],
  \\
  u_h^{n+1}
  &=
  \Lambda(w_h^{n+1}).
\end{aligned}
\label{eq:crkdg4-limited}
\end{equation}

\begin{remark}[{On $\mathcal{F}(u_h^{(1)}; \widetilde{u}_h^{(1)})$}]
The two-argument form $\mathcal{F}(u_h^{(1)}; \widetilde{u}_h^{(1)})$ is simply $\mathcal{F}(u_h^n)$, since $u_h^{(1)} = u_h^n$ and $\widetilde{u}_h^{(1)} = \Lambda_{\mathrm{tr}}(u_h^n; \overline{u}_h^n) = u_h^n|_{\partial\mathcal{T}_h}$ by the design of limiters.
\end{remark}

\begin{remark}[{Compactness}]
The proposed limiting strategy is particularly useful for preserving compactness when paired with the cRKDG method. In an SSP-RKDG method, each forward-Euler substep
uses the neighbor-coupled operator $\mathcal{F}$ and is followed by limiting,
so the dependence stencil grows through the internal stages. In the cRKDG scheme, the predictor stages rely solely on the local 
operator $\mathcal{G}$, while intercell coupling and trace limiting 
are deferred exclusively to the final corrector stage. This keeps the predictor part compact while preserving the
TVD/TVB-compatible structure needed for the cell-average update.
\end{remark}

\subsection{Applications to ADER-DG methods}\label{sec:ADER}
Although this paper primarily focuses on RKDG and cRKDG methods, the proposed predictor-corrector framework naturally accommodates the ADER-DG method, which is derived from a space-time weak formulation of the conservation law. Below, we discuss one version of this method that utilizes a local predictor \cite{gaburro2023high,dumbser2008unified,dumbser2008finite}.

The first step is to compute a predictor using a local space-time Galerkin method for \eqref{eq:scl}: find $z_h \in \mathbb{P}^{k'}(I_i \times [t^n,t^{n+1}])$, such that for all $w_h \in \mathbb{P}^{k'}(I_i\times[t^n,t^{n+1}])$, 
\begin{equation}\label{eq:ADER-I}
\int_{I_i} 
z_h(x,t^{n+1}) w_h(x,t^{n+1}) - u_h^n w_h(x,t^{n}) \, \mathrm{d} x
+\int_{t^n}^{t^{n+1}}\int_{I_i} 
-z_h(w_h)_t + f(z_h)_x  w_h \, \mathrm{d} x \, \mathrm{d} t = 0.
\end{equation}
Then, the corrector $u_h^{n+1} \in \mathbb{V}_h^k$ is updated via 
\begin{equation}\label{eq:ADER-II}
    u_h^{n+1} = u_h^n + \int_{t^n}^{t^{n+1}} \mathcal{F}(z_h(x,t)) \, \mathrm{d} t.
\end{equation}

Applying a quadrature rule on $[t^n,t^{n+1}]$ with nodes $\{t^{n,\ell}\}$ and weights $\{b_\ell\}$, one can rewrite \eqref{eq:ADER-II} in a form similar to \eqref{eq:rkdg-corrector-raw}, where $u_h^{(\ell)}(x) = z_h(x,t^{n,\ell})$. Consequently, \eqref{eq:ADER-I}, the relation $u_h^{(\ell)}(x) = z_h(x,t^{n,\ell})$, and \eqref{eq:proposed-trace-limiting}--\eqref{eq:proposed-post-limiter} constitute the trace-limited ADER-DG scheme. 
This scheme shares similar theoretical properties with its RKDG and cRKDG counterparts, and the details are omitted.

\section{Properties of Trace-Limited Schemes}
\label{sec:properties}

In this section, we establish the main properties of the trace-limited scheme \eqref{eq:generic-predictor}, \eqref{eq:proposed-trace-limiting}--\eqref{eq:proposed-post-limiter} for the 1D scalar equation \eqref{eq:scl} with periodic or compactly supported boundaries. 
Throughout the section, we assume a uniform mesh with cell size \(h\), and write $\lambda={\Delta t}/{h}.$

\subsection{Conservation}
\label{subsec:conservation}

We first note that the trace-limited corrector preserves the conservative
form of the DG method. Although the trace limiter modifies the left and right
traces entering the numerical flux, the numerical flux at each interface is
still single-valued. Therefore the flux leaving one cell is the same as the
flux entering the neighboring cell. This observation leads to the following proposition on conservation, whose  detailed proof is omitted. 

\begin{proposition}[Local and global conservation]
\label{prop:conservation}
The trace-limited scheme \eqref{eq:generic-predictor}, \eqref{eq:proposed-trace-limiting}--\eqref{eq:proposed-post-limiter} is locally
conservative. In particular, under periodic or compactly supported boundary conditions, $
\sum_i h\,\overline u_{h,i}^{n+1}
=
\sum_i h\,\overline u_{h,i}^{n}$.
\end{proposition}

\subsection{TVDM and TVBM properties}
\label{subsec:generalized-harten}

The key of the proposed method is that the TVD mechanism depends only
on the final cell-average update, not on how the predictor stages are
computed. The following lemma makes this precise. It shows that a
forward-Euler-type update is TVDM even if the raw traces
entering the limiter are arbitrary, provided that the limiter uses the old
cell averages as reference values. 
\begin{lemma}[General trace limiting satisfying the Harten form] 
\label{lem:trace-harten}
Let \(\widehat f\) be a monotone numerical flux with Lipschitz constants
\(L_1\) and \(L_2\) in its first and second arguments, respectively. For any DG functions $v\in \mathbb{V}_h^k$ and any reference cell averages $q=\overline{q}\in \mathbb{V}_h^0$ potentially irrelevant to $v$, define the TVD-limited traces $\widetilde v=\Lambda_{\mathrm{tr}}(v,\overline q)$ through \eqref{eq:trace-limiter-def} with $M = 0$ and $h_i \equiv h$. 
Consider the cell-average update
\begin{equation}
  \overline r_i
  =
  \overline q_i
  -\lambda
  \left[
  \widehat f(\widetilde{v}_{i+1/2}^-,\widetilde{v}_{i+1/2}^+)
  -
  \widehat f(\widetilde{v}_{i-1/2}^-,\widetilde{v}_{i-1/2}^+)
  \right].
  \label{eq:lemma-fe-update}
\end{equation}
If $\lambda\le {1}/{(2(L_1+L_2))}$, then
\eqref{eq:lemma-fe-update} admits the Harten form in Lemma \ref{lem:harten}, and thus $\operatorname{TVM}(r):=\sum_i |\overline{r}_{i+1}-\overline{r}_i|\leq \sum_i |\overline{q}_{i+1}-\overline{q}_i|:=\operatorname{TVM}(q)$.
\end{lemma}

\begin{proof}
Adding and subtracting
\(\widehat f(\widetilde{v}_{i+1/2}^-,\widetilde{v}_{i-1/2}^+)\), \eqref{eq:lemma-fe-update} can be rewritten as 
\begin{equation}
  \overline r_i
  =
  \overline q_i
  +C_{i+1/2}\Delta_+\overline q_i
  -D_{i-1/2}\Delta_-\overline q_i,
  \label{eq:lemma-harten-form}
\end{equation}
where
\begin{align}
  C_{i+1/2}
  ={}&
  -\lambda
  \frac{
  \widehat f(\widetilde{v}_{i+1/2}^-,\widetilde{v}_{i+1/2}^+)
  -
  \widehat f(\widetilde{v}_{i+1/2}^-,\widetilde{v}_{i-1/2}^+)
  }{
  \widetilde{v}_{i+1/2}^+-\widetilde{v}_{i-1/2}^+
  }\cdot
  \frac{
  \widetilde{v}_{i+1/2}^+-\widetilde{v}_{i-1/2}^+
  }{
  \Delta_+\overline q_i
  },
  \label{eq:C-coeff}
  \\
  D_{i-1/2}
  ={}&
  \lambda
  \frac{
  \widehat f(\widetilde{v}_{i+1/2}^-,\widetilde{v}_{i-1/2}^+)
  -
  \widehat f(\widetilde{v}_{i-1/2}^-,\widetilde{v}_{i-1/2}^+)
  }{
  \widetilde{v}_{i+1/2}^- - \widetilde{v}_{i-1/2}^-
  }\cdot
  \frac{
  \widetilde{v}_{i+1/2}^- - \widetilde{v}_{i-1/2}^-
  }{
  \Delta_-\overline q_i
  }.
  \label{eq:D-coeff}
\end{align}

Since \(\widehat f\) is nonincreasing in its second argument and Lipschitz
continuous with constant \(L_2\), we have
\begin{equation}
  0
  \le
  -
  \frac{
  \widehat f(\widetilde{v}_{i+1/2}^-,\widetilde{v}_{i+1/2}^+)
  -
  \widehat f(\widetilde{v}_{i+1/2}^-,\widetilde{v}_{i-1/2}^+)
  }{
  \widetilde{v}_{i+1/2}^+-\widetilde{v}_{i-1/2}^+
  }
  \le L_2.
  \label{eq:flux-ratio-second}
\end{equation}
It remains to estimate the ratios of limited traces. The trace \(\widetilde{v}_{i-1/2}^+\) is obtained from the cell
\(I_i\), while \(\widetilde{v}_{i+1/2}^+\) is obtained from the neighboring cell
\(I_{i+1}\). Thus the backward difference entering the limiter that defines
\(\widetilde{v}_{i+1/2}^+\) is
\begin{equation}
\Delta_-\overline q_{i+1}
  =
  \overline q_{i+1}-\overline q_i
  =
  \Delta_+\overline q_i.
\end{equation}
By the definition of the
minmod function, the quantities $\overline q_i-\widetilde{v}_{i-1/2}^+$ and $\overline q_{i+1}-\widetilde{v}_{i+1/2}^+$ have the same sign as \(\Delta_+\overline q_i\), whenever they are
nonzero, and their magnitudes do not exceed
\(|\Delta_+\overline q_i|\). Hence
\begin{equation}
  0
  \le
  \frac{
  \overline q_i-\widetilde{v}_{i-1/2}^+
  }{
  \Delta_+\overline q_i
  }
  \le 1,
  \qquad
  0
  \le
  \frac{
  \overline q_{i+1}-\widetilde{v}_{i+1/2}^+
  }{
  \Delta_+\overline q_i
  }
  \le 1.
  \label{eq:right-ratio-components}
\end{equation}
Therefore
\begin{equation}
\begin{aligned}
  \frac{
  \widetilde{v}_{i+1/2}^+-\widetilde{v}_{i-1/2}^+
  }{
  \Delta_+\overline q_i
  }
  ={}&
  1
  -
  \frac{
  \overline q_{i+1}-\widetilde{v}_{i+1/2}^+
  }{
  \Delta_+\overline q_i
  }
  +
  \frac{
  \overline q_i-\widetilde{v}_{i-1/2}^+
  }{
  \Delta_+\overline q_i
  } \in [0,2].
\end{aligned}
\label{eq:right-trace-ratio}
\end{equation}
Combining \eqref{eq:C-coeff},  \eqref{eq:flux-ratio-second} and \eqref{eq:right-trace-ratio} yields $0\le C_{i+1/2}\le 2\lambda L_2$. 
Similarly, one can prove that $0\le D_{i-1/2}\le 2\lambda L_1.$
Consequently, when $\lambda \leq 1/(2(L_1+L_2))$, we have $
  C_{i+1/2}+D_{i+1/2}
  \le
  2\lambda(L_1+L_2)
  \le 1.$
Harten's lemma then gives the TVDM property.
\end{proof}

\begin{remark}[{Arbitrary traces}]
The lemma is independent of how the raw traces
\(v_{i+1/2}^-\) and \(v_{i-1/2}^+\) are produced. They may be the traces
of the old solution, traces of a standard RKDG predictor stage, or traces of a
cRKDG predictor stage. Only the limited traces and the old reference cell
averages enter the TVDM argument.
\end{remark}

We now apply Lemma~\ref{lem:trace-harten} to the full trace-limited
RK corrector. 
The result uses only the final corrector
\eqref{eq:proposed-corrector} and it does not depend on whether the predictor
stages were computed with \(\mathcal F\) or \(\mathcal G\).

\begin{theorem}[TVDM]
\label{thm:proposed-tvdm}
Consider the trace-limited scheme
\eqref{eq:generic-predictor}, \eqref{eq:proposed-trace-limiting}--\eqref{eq:proposed-post-limiter} with the TVD
limiter \(M=0\). Assume that the corrector weights satisfy
\eqref{eq:positive-b-weights}
and that the monotone numerical flux \(\widehat f\) has Lipschitz constants
\(L_1\) and \(L_2\) in its first and second arguments, respectively. 
If $\lambda\le 1/{(2(L_1+L_2))}$, 
then $\TVM(u_h^{n+1})
  \le
  \TVM(u_h^n).$
\end{theorem}

\begin{proof}
Since the polynomial limiter \(\Lambda\) preserves cell averages, $\TVM(u_h^{n+1})=\TVM(w_h^{n+1})$. It suffices to prove that \(\TVM(w_h^{n+1})\le \TVM(u_h^n)\). Indeed, taking cell averages in the corrector \eqref{eq:proposed-corrector} and using
\(\sum_{\ell=1}^s b_\ell=1\), we have
\begin{equation}
  \overline w_{h,i}^{n+1}
  =
  \sum_{\ell=1}^s b_\ell
  \Bigl(
  \overline u_{h,i}^n
  -\lambda
  \Bigl[
  \widehat f
  \left(
  \widetilde u_{h,i+1/2}^{(\ell),-},
  \widetilde u_{h,i+1/2}^{(\ell),+}
  \right) -
  \widehat f
  \left(
  \widetilde u_{h,i-1/2}^{(\ell),-},
  \widetilde u_{h,i-1/2}^{(\ell),+}
  \right)
  \Bigr]
  \Bigr):=\sum_{\ell=1}^s b_\ell \overline{r}_{h,i}^{(\ell)}.
\label{eq:convex-corrector-avg}
\end{equation}
For each fixed stage \(\ell\), the expression defining $\overline{r}_{h,i}^{(\ell)}$ is a
forward-Euler update of the form \eqref{eq:lemma-fe-update}. 
By construction, the trace limiter uses the old cell averages
\(\overline{u}_h^n\) as reference, namely $\overline q = \overline u_h^n$. Hence Lemma~\ref{lem:trace-harten}
applies to each stage contribution, giving $\TVM(r_h^{(\ell)})
  \le
  \TVM(u_h^n)$.

By \eqref{eq:positive-b-weights} and the convexity of the TVM seminorm, we have
\begin{equation}
    \TVM(w_h^{n+1})
  \le
  \sum_{\ell=1}^s b_\ell \TVM(r_h^{(\ell)})
  \le
  \sum_{\ell=1}^s b_\ell \TVM(u_h^n)
  =
  \TVM(u_h^n).
\end{equation}
This completes the proof.
\end{proof}
\begin{remark}[Role of $\Lambda$]
From the proof, one can see that the TVDM property can still be preserved even without the polynomial limiter $\Lambda$. 
However, the resulting DG polynomial may still contain significant higher-order oscillations, although its cell averages remain TVD. 
Thus, \(\Lambda\) is needed for practical robustness.
\end{remark}
\begin{remark}[TVBM]
\label{rem:proposed-tvbm}
Similar to the SSP-RKDG method, for \(M>0\), the TVB limiter no longer guarantees strict TVDM. The modified minmod may accept the first argument when its magnitude is below \(Mh^2\), even if it violates TVD bounds. This deviation is controlled by the \(Mh^2\) threshold and yields a standard TVBM estimate over finite time, as in classical TVB SSP-RKDG schemes.
\end{remark}

\subsection{Non-activation in smooth regions}
\label{subsec:nonactivation}
Finally, we address the accuracy of the limiting procedure. 
Under a mild
assumption on the predictor stages and a CFL condition, we show that, for sufficiently small \(h\),
the trace limiter remains inactive in smooth monotone regions away from critical
points, and also remains inactive near smooth extrema when the TVB constant
\(M\) is sufficiently large.
Consequently, the limiting procedure preserves the formal order of accuracy of the underlying RKDG or cRKDG schemes.

Recall that \(c_\ell\) denotes the Butcher node for stage \(\ell\), with \(c_\ell=\sum_{j=1}^{\ell-1} a_{\ell j}\).

\begin{theorem}[Non-activation of trace limiter]
\label{thm:nonactivation}
Consider one step of the trace-limited scheme
\eqref{eq:generic-predictor}, \eqref{eq:proposed-trace-limiting}--\eqref{eq:proposed-post-limiter}. 
Denote by $\mathcal N_i:=I_{i-1}\cup I_i\cup I_{i+1}$ a fixed stencil neighborhood of $I_i$. We assume that the following conditions hold:
\begin{enumerate}
\item[\textup{(A1)}]
The exact solution \(u(x,t)\) is smooth for \((x,t)\) in the neighborhood $\mathcal{N}_i$
over the time interval \([t^n,t^{n+1}]\).

\item[\textup{(A2)}]
The CFL number satisfies
\begin{equation}
  |c_\ell|\,\lambda\,\max_u |f'(u)|<\tfrac12,\qquad \ell=1,\ldots,s.
  \label{eq:nonactivation-cfl}
\end{equation}

\item[\textup{(A3)}]
The one-step approximation error of the predictor stages satisfies
\begin{equation}
  \big\|
  u_h^{(\ell)}
  -
  u(\cdot,t^n+c_\ell\Delta t)
  \big\|_{L^\infty(\mathcal{N}_i)}
  \le C_\ell h^2,\qquad \ell=1,\ldots,s,
  \label{eq:stage-accuracy-assumption}
\end{equation}
for constants \(C_\ell\) independent of \(h\).
\end{enumerate}
Then the following hold for all sufficiently small \(h\):
\begin{enumerate}
\item[\textup{(i)}] \emph{(TVD)} If \(|(u_x)_i^n|\ge\kappa\) for some constant
\(\kappa>0\) independent of \(h\), then the TVD trace limiter \((M=0)\) does not
modify the traces of \(u_h^{(\ell)}\) on \(I_i\).

\item[\textup{(ii)}] \emph{(TVB)} There exists a constant \(M_0>0\), depending only
on \(u\) and on the parameters in \textup{(A2)}--\textup{(A3)}, but independent of
\(h\), such that for every \(M\ge M_0\) the TVB trace limiter does not
modify the traces of \(u_h^{(\ell)}\) on \(I_i\).
\end{enumerate}
\end{theorem}
\begin{proof}
We prove the statement for the right trace; the left trace is analogous.  Set $\delta_{i,+}^{(\ell)}
  = u_{h,i+1/2}^{(\ell),-}-\overline u_{h,i}^n.$
Recall the definition of the trace limiter in \eqref{eq:trace-limiter-def}. The trace limiter activates to modify the right trace on $I_i$ exactly 
when $\widetilde m$ fails to return $\delta_{i,+}^{(\ell)}$ upon comparison 
with the cell-average differences $\Delta_\pm \overline{u}_{h,i}^n$.

Let \(v_i^n = v(x_i,t^n)\) with \(v=u,u_t,u_x\). 
Using the stage accuracy assumption
\eqref{eq:stage-accuracy-assumption} for $\ell=1$ and Taylor expansions, $\Delta_\pm \overline{u}_{h,i}^n$ admit the expressions:
\begin{equation}
  \Delta_+\overline u_{h,i}^n
  =
  (u_x)_i^n h+O(h^2)\quad\text{and}\quad
  \Delta_-\overline u_{h,i}^n
  =
  (u_x)_i^n h+O(h^2).
  \label{eq:mean-diff-expansion}
\end{equation}

For $\delta_{i,+}^{(\ell)}$, using the stage accuracy assumption
\eqref{eq:stage-accuracy-assumption}, Taylor expansions, and \(u_t = -f'(u)u_x\)
with \(\Delta t = \lambda h\), we obtain
\begin{equation}
\begin{aligned}
  \delta_{i,+}^{(\ell)}
   =& u(x_{i+1/2},t^n+c_\ell \Delta t) - \frac{1}{h}\int_{I_i} u(x,t^n) \mathrm{d}x + \mathcal{O}(h^2)\\
   =&u_i^n +\tfrac12(u_x)_{i}^{n}h + (u_t)_i^n c_\ell \Delta t - \frac{1}{h}\int_{I_i} \left[u_i^n + (u_x)_i^n (x-x_i)\right] \mathrm{d}x + \mathcal{O}(h^2)\\
  =& \theta_\ell
  (u_x)_i^n h
  +O(h^2), \qquad \text{with }\; \theta_\ell := 
  \tfrac12-c_\ell\lambda f'(u_i^n).
  \end{aligned}
  \label{eq:stage-face-expansion}
\end{equation}
By (A2), the leading coefficient \(\theta_\ell\)
lies strictly between \(0\) and \(1\), uniformly over the stages, i.e. $0<\theta_{\min}\le \theta_\ell\le \theta_{\max}<1$.

\emph{(i)} Suppose \(|(u_x)_i^n|\ge\kappa\) for a constant \(\kappa>0\) independent
of \(h\). Then \((u_x)_i^n h\) is the genuine leading term in
\eqref{eq:mean-diff-expansion}--\eqref{eq:stage-face-expansion}.
Thus, for \(h\) sufficiently small,
\(\delta_{i,+}^{(\ell)}\), \(\Delta_+\overline u_{h,i}^n\), and
\(\Delta_-\overline u_{h,i}^n\) share the sign of \((u_x)_i^n\). With $0<\theta_\ell<1$, \(|\delta_{i,+}^{(\ell)}|\) is smaller than both
\(|\Delta_+\overline u_{h,i}^n|\) and \(|\Delta_-\overline u_{h,i}^n|\). 
The minmod function therefore returns its first argument \(\delta_{i,+}^{(\ell)}\).

\emph{(ii)} Consider any cell \(I_i\) in the smooth region. Let the \(O(h^2)\) remainders in
\eqref{eq:mean-diff-expansion}--\eqref{eq:stage-face-expansion} be bounded by
\(C_\ast h^2\), uniformly over the stages. 
Choose $M_0=\theta_{\max}K+C_\ast$ with $K={2C_\ast}/{\min\{\theta_{\min},1-\theta_{\max}\}}$. 
We will show that any \(M\ge M_0\) is sufficient.

Indeed, if \(|\delta_{i,+}^{(\ell)}|\le Mh^2\), then \(\widetilde m\) returns
\(\delta_{i,+}^{(\ell)}\) directly by \eqref{eq:tvb-minmod}, and the right trace
is unchanged. Otherwise, \(|\delta_{i,+}^{(\ell)}|>Mh^2\). From
\eqref{eq:stage-face-expansion}, we have $|\delta_{i,+}^{(\ell)}|
  \le
  \theta_{\max}|(u_x)_i^n|h+C_\ast h^2.$
Thus $|(u_x)_i^n|
  >
  ({(M-C_\ast)}/{\theta_{\max}})h
  \ge Kh.$
The definition of \(K\) then implies
$\theta_{\min}\,|(u_x)_i^n|\,h>C_\ast h^2$ and $|(u_x)_i^n|\,h>C_\ast h^2$.
Therefore, the leading terms proportional to \((u_x)_i^n h\) dominate the
\(O(h^2)\) remainders in \(\delta_{i,+}^{(\ell)}\),
\(\Delta_+\overline u_{h,i}^n\), and
\(\Delta_-\overline u_{h,i}^n\). Hence all three quantities have the same sign as $(u_x)_i^n$.
Moreover,
\begin{equation}
  |\Delta_\pm\overline u_{h,i}^n|-|\delta_{i,+}^{(\ell)}|
  \ \ge\
  (1-\theta_{\max})|(u_x)_i^n|\,h-2C_\ast h^2
  \ >\ 0.
\end{equation}
The minmod function therefore returns its first argument \(\delta_{i,+}^{(\ell)}\).

The left trace is treated identically, with \(\theta_\ell\) replaced by
\(1/2+c_\ell\lambda f'(u_i^n)\). 
\end{proof}

Theorem~\ref{thm:nonactivation} should be read as an accuracy-preservation
statement for the limiter, not as a standalone convergence theorem for the
full RKDG or cRKDG method. It says that if the predictor stages already have
the expected local accuracy in smooth regions, then the trace limiter is
inactive there for sufficiently fine meshes. The post-step polynomial limiter
\(\Lambda\) satisfies the analogous standard TVB non-activation property when
applied to the smooth candidate \(w_h^{n+1}\) \cite{rkdg2}.

\begin{remark}[{CFL constraint}]\label{rem:accuracy-consequence}
Assumption~(A2) is mild. It imposes a CFL restriction depending on
$\max_\ell |c_\ell|$, the largest absolute Butcher node. 
This restriction is 
typically much less stringent than the CFL constraints required for linear stability. 
Indeed, for the second-order midpoint, 
third-order Heun \eqref{eq:crkdg3-limited}, and classical fourth-order 
\eqref{eq:crkdg4-limited} schemes, we have $\max_\ell |c_\ell| = 1/2, 2/3, 1$, respectively. Writing $\mathrm{CFL} = \lambda \max_u |f'(u)|$, 
the CFL limits implied by \eqref{eq:nonactivation-cfl} are $1$, $3/4$, and $1/2$, 
respectively. 
In comparison, the corresponding linear stability constraints are are $0.333$, $0.209$, and $0.145$ for RKDG schemes, and $0.333$, $0.178$, and $0.103$ for 
cRKDG schemes, respectively, when pairing the time integrator with polynomials of one degree lower. 
Moreover, for any consistent numerical flux, one can show that $\lambda \max_u |f'(u)|\leq 1/2$ is a necessary condition for the CFL bound $1/(2(L_1+L_2))$ in Lemma \ref{lem:trace-harten}.
Therefore, when
$\max_\ell |c_\ell|\leq 1$, satisfying the CFL condition in
Lemma~\ref{lem:trace-harten} almost implies \eqref{eq:nonactivation-cfl}.
\end{remark}

\begin{remark}[{Accuracy assumption}]
Assumption~(A3) is mild. It requires the one-step local error to be second-order accurate, while effectively prescribing only first-order global temporal accuracy. Such an assumption can be readily verified in the context of finite difference approximations. It is also consistent with the \(L^2\) error estimate if one follows the derivation in \cite{ai20222} when $k\geq 2$.
\end{remark}

\section{Extensions to Systems and Multi-Dimensions}
\label{sec:extensions}

While previous sections focused on 1D scalar conservation laws, we now extend the method to systems and multidimensional problems. The time discretization and trace-limited corrector structure remain unchanged; what changes is only the application of the trace limiter $\Lambda_{\mathrm{tr}}$ to vector-valued or multidimensional DG polynomials. While our focus here is on the 2D case, this extension applies naturally to higher dimensions.

\subsection{1D systems}
\label{subsec:oned-systems}

Consider the 1D system
\begin{equation}
  \mathbf u_t+\mathbf f(\mathbf u)_x=0,
  \qquad
  \mathbf u\in\mathbb R^m.
  \label{eq:system-1d}
\end{equation}
The trace-limited RKDG/cRKDG scheme has the same form as that in
Section~\ref{subsec:proposed-scheme}, with all scalar quantities replaced by
vectors:
\begin{equation}
\label{eq:system-scheme-1d}
\begin{aligned}
  \mathbf u_h^{(\ell)}
  =&
  \mathbf u_h^n
  +\Delta t\sum_{j=1}^{\ell-1}a_{\ell j}\mathcal H(\mathbf u_h^{(j)}),\qquad \ell=1,\ldots,s,
  \qquad
  \mathcal H\in\{\mathcal F,\mathcal G\},\\
  \widetilde{\mathbf u}_h^{(\ell)}
  =&
  \Lambda_{\mathrm{tr}}
  \left(\mathbf u_h^{(\ell)};\overline{\mathbf u}_h^n\right),
  \qquad
  \ell=1,\ldots,s.\\
  \mathbf w_h^{n+1}
  =&
  \mathbf u_h^n
  +\Delta t\sum_{\ell=1}^s b_\ell
  \mathcal F(\mathbf u_h^{(\ell)};\widetilde{\mathbf u}_h^{(\ell)}),\\
  \mathbf u_h^{n+1}=&\Lambda(\mathbf w_h^{n+1}).
\end{aligned}
\end{equation}

We now define the trace limiter $\Lambda_{\mathrm{tr}}:[\mathbb{V}_h^k]^m\times[\mathbb{V}_h^0]^m
\to [\mathbb{T}_h^k]^m$.
Throughout this section, for a given pair
\((\mathbf v,\overline{\mathbf q})\in [\mathbb{V}_h^k]^m\times[\mathbb{V}_h^0]^m\), we write
\begin{equation}
\widetilde{\mathbf v}
=
\Lambda_{\mathrm{tr}}(\mathbf v;\overline{\mathbf q})
\end{equation}
for the limited trace function. On a single mesh cell \(I_i\), the quantities
\(\mathbf v_{i+1/2}^-\) and \(\mathbf v_{i-1/2}^+\) denote the raw traces to be limited, while
\(\overline{\mathbf q}_i\) denotes the reference cell average used by the limiter.
Define the trace deviations from the reference average by
\begin{equation}
  \boldsymbol \delta_{i,+}
  =
  \mathbf v_{i+1/2}^{-}-\overline{\mathbf q}_i,
  \qquad
  \boldsymbol \delta_{i,-}
  =
  \overline{\mathbf q}_i-\mathbf v_{i-1/2}^{+}.
  \label{eq:system-raw-deviations-1d}
\end{equation}
The neighboring mean differences are defined by 
\begin{equation}
    \Delta_+ \overline{\mathbf q}_{i}= \overline{\mathbf q}_{i+1}-\overline{\mathbf q}_{i}, 
  \qquad 
  \Delta_- \overline{\mathbf q}_{i}= \overline{\mathbf q}_{i}-\overline{\mathbf q}_{i-1}.
\end{equation}

For systems, the limiting is performed in local characteristic variables. 
We need to use the eigen-decomposition $
  A_i =\mathbf f'(\overline{\mathbf q}_i)
  =
  R_i D_i R_i^{-1}$,
where \(D_i\) is a diagonal matrix of eigenvalues and the columns of \(R_i\) are the corresponding right eigenvectors.
These quantities are transformed into
characteristic variables:
\begin{equation}
\begin{aligned}
  \boldsymbol\alpha_{i,+}
  =R_i^{-1}\boldsymbol \delta_{i,+},
  \quad 
  \boldsymbol\alpha_{i,-}
  =R_i^{-1}\boldsymbol \delta_{i,-}, \quad 
  \boldsymbol\alpha_{i}^{+}
  =R_i^{-1}\Delta_+ \overline{\mathbf q}_{i}, \quad 
  \boldsymbol\alpha_{i}^{-}
  =R_i^{-1}\Delta_-\overline{\mathbf q}_{i}.
\end{aligned}
\label{eq:system-char-differences-1d}
\end{equation}
The scalar TVB-modified minmod function is then applied componentwise:
\begin{equation}
\begin{aligned}
  \widetilde{\boldsymbol\alpha}_{i,+}
  &=
  \widetilde m
  \left(
  \boldsymbol\alpha_{i,+},
  \boldsymbol\alpha_{i}^{+},
  \boldsymbol\alpha_{i}^{-};\mathbf{M},h_i
  \right),\qquad 
  \widetilde{\boldsymbol\alpha}_{i,-}
  =
  \widetilde m
  \left(
  \boldsymbol\alpha_{i,-},
  \boldsymbol\alpha_{i}^{+},
  \boldsymbol\alpha_{i}^{-};\mathbf{M},h_i
  \right).
\end{aligned}
\label{eq:system-char-minmod-1d}
\end{equation}
The TVB parameter $\mathbf{M}\in\mathbb{R}^{m}$ may also be chosen componentwise in characteristic variables.
The limited traces are defined by transforming back the modified values to  conservative variables and adding them to the cell averages:
\begin{equation}
  \widetilde{\mathbf v}_{i+1/2}^{-}
  =
  \overline{\mathbf q}_i+R_i\widetilde{\boldsymbol\alpha}_{i,+},
  \qquad
  \widetilde{\mathbf v}_{i-1/2}^{+}
  =
  \overline{\mathbf q}_i-R_i\widetilde{\boldsymbol\alpha}_{i,-}.
  \label{eq:system-limited-traces-1d}
\end{equation}
This completes the definition of trace limiting $\widetilde{\mathbf v} = \Lambda_\mathrm{tr}(\mathbf v, \overline{\mathbf q})$.

The polynomial limiter \(\Lambda\) for systems is defined in the standard way. 
When limiting a polynomial \(\mathbf v\), the reference average in the trace limiter is set as its own cell average, i.e. \(\overline{\mathbf q}=\overline{\mathbf v}\).
A cell \(I_i\) is called troubled if at least one of its traces is modified by
\(\Lambda_{\mathrm{tr}}(\mathbf v;\overline{\mathbf v})\). 
On nontroubled cells, \(\Lambda\) leaves the polynomial \(\mathbf v|_{I_i}\) unchanged. 
On a troubled cell \(I_i\), \(\Lambda\) replaces \(\mathbf{v}|_{I_i}\) with a polynomial in \([\mathbb{P}^{\min(k,2)}(I_i)]^m\) that preserves the cell average \(\overline{\mathbf{v}}_i\) and matches the limited interface traces \(\widetilde{\mathbf{v}}_{i+1/2}^-\) and \(\widetilde{\mathbf{v}}_{i-1/2}^+\).

\subsection{2D scalar equations}
\label{subsec:twod-scalar}
We next consider a scalar conservation law in 2D,
\begin{equation}
  u_t+f(u)_x+g(u)_y=0,\qquad (x,y)\in\Omega,\quad t>0 .
  \label{eq:scalar-2d}
\end{equation}
Let \(\mathcal T_h\) be a Cartesian mesh with cells
\(K_{ij}=I_i\times J_j=[x_{i-1/2},x_{i+1/2}]
\times[y_{j-1/2},y_{j+1/2}]\). We set
\(h_i^x=x_{i+1/2}-x_{i-1/2}\),
\(h_j^y=y_{j+1/2}-y_{j-1/2}\), and introduce the reference coordinates
\(\xi=2(x-x_i)/h_i^x\), \(\eta=2(y-y_j)/h_j^y\), with
\((\xi,\eta)\in[-1,1]^2\), on each cell \(K_{ij}\).
The DG space is
\(\mathbb V_h^k=\{v\in L^2(\Omega):v|_{K_{ij}}\in \mathbb Z^k(K_{ij})
\ \text{for all } i,j\}\), where \(\mathbb Z^k(K_{ij})\) denotes either the
total-degree space \(\mathbb P^k(K_{ij})\) or the tensor-product space
\(\mathbb Q^k(K_{ij})\).

We now define the trace limiter
\(\Lambda_{\mathrm{tr}}:\mathbb V_h^k\times\mathbb V_h^0\to\mathbb T_h^k\), where
\(\mathbb T_h^k=\{v|_{\partial\mathcal T_h}:v\in\mathbb V_h^k\}\) and
\(\partial\mathcal T_h=\{\partial K:K\in\mathcal T_h\}\). 
For a given pair \((v,\overline q)\in \mathbb V_h^k\times \mathbb V_h^0\), we write
\(\widetilde v=\Lambda_{\mathrm{tr}}(v;\overline q)\) for the limited trace function.
Let the affine part (the \(\mathbb P^1\)-component) of \(v\) on \(K_{ij}\) be
\begin{equation}
  \Pi_1 v|_{K_{ij}}
  =
  \overline v_{ij}
  +
  v^x_{ij}\xi
  +
  v^y_{ij}\eta .
  \label{eq:P1-projection-2d}
\end{equation}
The edge-average deviations of this affine part relative to
\(\overline q_{ij}\) are
\begin{equation}
\delta_R=(\overline v_{ij}+v^x_{ij})-\overline q_{ij},\;
  \delta_L=\overline q_{ij}-(\overline v_{ij}-v^x_{ij}),\;
  \delta_T=(\overline v_{ij}+v^y_{ij})-\overline q_{ij},\;
  \delta_B=\overline q_{ij}-(\overline v_{ij}-v^y_{ij}),
\end{equation}
corresponding to the right, left, top, and bottom edges $\partial K_{ij}^{E}, E\in\{R,L,T,B\}$, respectively.
The limited deviations are then defined by
\begin{subequations}
\label{eq:2D-delta}
\begin{align}
  \widetilde\delta_E
  &=
  \widetilde m
  \left(
  \delta_E,
  \overline q_{i+1,j}-\overline q_{ij},
  \overline q_{ij}-\overline q_{i-1,j};
  M,h_i^x
  \right),
  \qquad E\in\{R,L\},
  \label{eq:2D-deltaRL}
  \\
  \widetilde\delta_E
  &=
  \widetilde m
  \left(
  \delta_E,
  \overline q_{i,j+1}-\overline q_{ij},
  \overline q_{ij}-\overline q_{i,j-1};
  M,h_j^y
  \right),
  \qquad E\in\{T,B\}.
  \label{eq:2D-deltaTB}
\end{align}
\end{subequations}

If none of the four deviations is modified, then
\(\Lambda_{\mathrm{tr}}(v;\overline q)\) returns the original inner traces of \(v\) on
\(\partial K_{ij}\). 
Otherwise, construct the limited affine function
\begin{equation}
  \widetilde v_{ij}(\xi,\eta)
  =
  \overline q_{ij}
  +
  \frac12(\widetilde\delta_R+\widetilde\delta_L)\xi
  +
  \frac12(\widetilde\delta_T+\widetilde\delta_B)\eta .
  \label{eq:modaffine}
\end{equation}
For cells with active limiting, the trace limiter returns four inner edge traces of \eqref{eq:modaffine}:
\begin{equation}
  \Lambda_{\mathrm{tr}}(v;\overline q)|_{\partial K_{ij}^{E}}
  =
  \widetilde v_{ij}|_{\partial K_{ij}^{E}},
  \qquad E\in\{R,L,T,B\}.
  \label{eq:Lambda-trace-2d-scalar}
\end{equation}

The polynomial limiter \(\Lambda:\mathbb V_h^k\to\mathbb V_h^k\) is the classical Cockburn--Shu limiter in multiple dimensions on Cartesian meshes \cite{rkdg5}. In the above construction, it corresponds to taking \(\overline q=\overline v\). Thus, on a cell \(K_{ij}\), \((\Lambda v)|_{K_{ij}}=v|_{K_{ij}}\) if the limiting in \eqref{eq:2D-delta} is inactive, while \((\Lambda v)|_{K_{ij}}=\widetilde v_{ij}\) otherwise, where \(\widetilde v_{ij}\) is defined by \eqref{eq:modaffine} with \(\overline q_{ij}=\overline v_{ij}\).

The next proposition shows that the TVB-modified minmod limiter is idempotent on active cells; re-applying it to the resulting edge averages leaves them unchanged.
\begin{proposition}[Edge-average TVB property on active cells]
\label{prop:2d-edge-average-tvb}
Suppose that at least one of the four edge-average deviations
\(\delta_R,\delta_L,\delta_T,\delta_B\) is modified, so that
\(\Lambda_{\mathrm{tr}}(v;\overline{q})\) returns the traces of the
limited affine reconstruction \(\widetilde v_{ij}\). Define the edge-average deviations of \(\widetilde v_{ij}\) against $\overline{q}_{ij}$ by
\begin{equation}\label{eq:2Dedge_deviations}
\begin{aligned}
  e_R&=\frac{1}{h_j^y}\int_{\partial K_{ij}^{R}}\widetilde v_{ij}\,\mathrm{d}y
  -\overline q_{ij},
  &\qquad
  e_L&=
  \overline q_{ij}
  -\frac{1}{h_j^y}\int_{\partial K_{ij}^{L}}\widetilde v_{ij}\,\mathrm{d}y,
  \\
  e_T&=\frac{1}{h_i^x}\int_{\partial K_{ij}^{T}}\widetilde v_{ij}\,\mathrm{d}x
  -\overline q_{ij},
  &\qquad
  e_B&=
  \overline q_{ij}
  -\frac{1}{h_i^x}\int_{\partial K_{ij}^{B}}\widetilde v_{ij}\,\mathrm{d}x.
\end{aligned}
\end{equation}
Then these deviations are fixed points of the corresponding TVB-modified minmod operations with respect to the reference cell averages \(\overline q\), i.e.,
\begin{subequations}
\begin{align}
   e_E
  =
  \widetilde m
  \left(
  e_E,
  \overline q_{i+1,j}-\overline q_{ij},
  \overline q_{ij}-\overline q_{i-1,j};
  M,h_i^x
  \right),
  \qquad E\in\{R,L\},\label{eq:eReL}\\
  e_E
  =
  \widetilde m
  \left(
  e_E,
  \overline q_{i,j+1}-\overline q_{ij},
  \overline q_{ij}-\overline q_{i,j-1};
  M,h_j^y
  \right),
  \qquad E\in\{T,B\}.\label{eq:eTeB}
\end{align}
\end{subequations}
\end{proposition}
\begin{proof}
We only prove \eqref{eq:eReL}; the proof of \eqref{eq:eTeB} is the same. 
Set
\begin{equation}
D_+^x=\overline q_{i+1,j}-\overline q_{ij},\qquad
  D_-^x=\overline q_{ij}-\overline q_{i-1,j},\qquad
  K=M(h_i^x)^2.
\end{equation}
We first note that the fixed-point set of the map
\(a\mapsto \widetilde m(a,D_+^x,D_-^x;M,h_i^x)\) is an interval. Indeed, if
\(D_+^x\) and \(D_-^x\) are both positive, this set is
\(
  [-K,\max\{K,\min(D_+^x,D_-^x)\}].
\)
If \(D_+^x\) and \(D_-^x\) are both negative, this set is
\(
  [\min\{-K,\max(D_+^x,D_-^x)\},K].
\)
If their signs differ, the fixed-point set is simply \([-K,K]\).

By construction, \(\widetilde\delta_R\) and \(\widetilde\delta_L\) are fixed points of this map. Since the fixed-point set is an interval, their average also belongs to this interval and is thus a fixed point. Moreover, by \eqref{eq:modaffine} and \eqref{eq:2Dedge_deviations}, we have 
$e_R=e_L=(\widetilde\delta_R+\widetilde\delta_L)/2$.
Therefore \(e_R\) and \(e_L\) satisfy \eqref{eq:eReL}.
\end{proof}

\subsection{2D systems}
\label{subsec:twod-systems}

For a 2D system
\begin{equation}
  \mathbf u_t+\mathbf f(\mathbf u)_x+\mathbf g(\mathbf u)_y=0,
  \qquad
  \mathbf u\in\mathbb R^m,
  \label{eq:system-2d}
\end{equation}
the limiters are applied in local characteristic variables, separately in the
two coordinate directions. The construction combines the characteristic
limiting in Section~\ref{subsec:oned-systems} with the edge-average affine
reconstruction in Section~\ref{subsec:twod-scalar}.

We now define the trace limiter $\Lambda_{\mathrm{tr}}:[\mathbb V_h^k]^m\times[\mathbb V_h^0]^m
  \to[\mathbb T_h^k]^m.$
For \(\mathbf v\in[\mathbb V_h^k]^m\) and
\(\overline{\mathbf q}\in[\mathbb V_h^0]^m\), write
\(\widetilde{\mathbf v}=\Lambda_{\mathrm{tr}}(\mathbf v;\overline{\mathbf q})\).
On a cell \(K_{ij}\), consider the eigen-decompositions
\begin{equation}
A^x_{ij}
  =
  \mathbf f'(\overline{\mathbf q}_{ij})
  =
  R^x_{ij}D^x_{ij}(R^x_{ij})^{-1},
  \qquad
  A^y_{ij}
  =
  \mathbf g'(\overline{\mathbf q}_{ij})
  =
  R^y_{ij}D^y_{ij}(R^y_{ij})^{-1},
\end{equation}
where the columns of \(R^x_{ij}\) and \(R^y_{ij}\) are the corresponding right
eigenvectors. Write the affine part of \(\mathbf v\) on \(K_{ij}\) as
\begin{equation}
  \Pi_1\mathbf v|_{K_{ij}}
  =
  \overline{\mathbf v}_{ij}
  +
  \mathbf v^x_{ij}\xi
  +
  \mathbf v^y_{ij}\eta .
  \label{eq:P1-projection-2d-system}
\end{equation}
The edge-average deviations relative to \(\overline{\mathbf q}_{ij}\) are
\begin{equation}    
\begin{aligned}
  \boldsymbol\delta_R&=(\overline{\mathbf v}_{ij}+\mathbf v^x_{ij})
  -\overline{\mathbf q}_{ij},
  &\quad
  \boldsymbol\delta_L&=\overline{\mathbf q}_{ij}
  -(\overline{\mathbf v}_{ij}-\mathbf v^x_{ij}),\\
  \boldsymbol\delta_T&=(\overline{\mathbf v}_{ij}+\mathbf v^y_{ij})
  -\overline{\mathbf q}_{ij},
  &\quad
  \boldsymbol\delta_B&=\overline{\mathbf q}_{ij}
  -(\overline{\mathbf v}_{ij}-\mathbf v^y_{ij}).
\end{aligned}
\end{equation}
The neighboring mean differences in $x$- and $y$-directions are defined by
\begin{equation}
\Delta_\pm^x\overline{\mathbf q}_{ij}
  =
  \pm(\overline{\mathbf q}_{i\pm1,j}-\overline{\mathbf q}_{ij}),
  \qquad
  \Delta_\pm^y\overline{\mathbf q}_{ij}
  =
  \pm(\overline{\mathbf q}_{i,j\pm1}-\overline{\mathbf q}_{ij}).
\end{equation}

For \(d=x,y\), set $\mathcal E_x=\{R,L\}$ and $\mathcal E_y=\{T,B\}$.
For each edge \(E\in\mathcal E_d\), project the corresponding edge deviation and
neighboring mean differences into the characteristic variables of \(A^d_{ij}\):
\begin{equation}
\boldsymbol\alpha_E^d=(R^d_{ij})^{-1}\boldsymbol\delta_E,
  \qquad
  \boldsymbol\alpha_\pm^d=(R^d_{ij})^{-1}
  \Delta_\pm^d\overline{\mathbf q}_{ij}.
\end{equation}
The scalar TVB-modified minmod function is then applied componentwise, and the
result is transformed back to conservative variables:
\begin{equation}
\widetilde{\boldsymbol\delta}_E
  =
  R^d_{ij}\,
  \widetilde m
  \left(
  \boldsymbol\alpha_E^d,
  \boldsymbol\alpha_+^d,
  \boldsymbol\alpha_-^d;
  \mathbf{M},h^d_{i/j}
  \right),
  \quad 
  E\in\mathcal E_d,\quad 
  d=x,y, \quad  h_{i/j}^d = h_i^x\text{ or }
h_j^y.
\end{equation}

If the componentwise characteristic limiting does not modify any of the four
edge deviations, then \(\Lambda_{\mathrm{tr}}(\mathbf v;\overline{\mathbf q})\)
returns the original inner traces of \(\mathbf v\). Otherwise, construct
\begin{equation}
  \widetilde{\mathbf v}_{ij}(\xi,\eta)
  =
  \overline{\mathbf q}_{ij}
  +
  \frac12(\widetilde{\boldsymbol\delta}_R+
  \widetilde{\boldsymbol\delta}_L)\xi
  +
  \frac12(\widetilde{\boldsymbol\delta}_T+
  \widetilde{\boldsymbol\delta}_B)\eta.
  \label{eq:system-2d-affine-reconstruction}
\end{equation}
On an active cell, the trace limiter returns the inner edge traces of
\(\widetilde{\mathbf v}_{ij}\).

The polynomial limiter $\Lambda:[\mathbb V_h^k]^m\to[\mathbb V_h^k]^m$  leaves \(\mathbf v|_{K_{ij}}\) unchanged on nontroubled cells, while it replaces \(\mathbf v|_{K_{ij}}\) by the limited affine reconstruction
\(\widetilde{\mathbf v}_{ij}\) with 
\(\overline{\mathbf q}=\overline{\mathbf v}\) on troubled cells. This retrieves the usual characteristic-wise
Cockburn--Shu limiter for systems on Cartesian meshes \cite{rkdg5}.

\begin{remark}[{Theoretical considerations}]
We view the proposed methods in this section as a natural algorithmic extension of the 1D scalar construction. A rigorous TVDM/TVBM analysis is currently unavailable for high-dimensional nonlinear systems. Notably, this aligns with the theoretical boundaries of the classical TVDM/TVBM framework, as well as the intrinsic analytical difficulties of the underlying PDEs in these complex settings. Nevertheless, properties such as local and global conservation carry over from the 1D case, and the practical effectiveness of these algorithms is demonstrated by numerical tests.
\end{remark}

\section{Numerical Tests}\label{sec:numerical_tests}
We test the proposed trace-limited cRKDG schemes on a suite of scalar conservation laws and the Euler equations, and compare the results with those obtained using classical SSP-RKDG schemes of the same order. {To save space, we present only the results for third- and fourth-order schemes in the accuracy tests, and for third-order schemes in the shock-capturing tests. 
Besides \eqref{eq:ssp-rk3-dg}, \eqref{eq:crkdg3-limited} and \eqref{eq:crkdg4-limited}, we use the Spiteri--Ruuth scheme in \cite{spiteri2002new} for the fourth-order SSP-RKDG method.
All schemes use the local Lax--Friedrichs numerical flux, with the dissipation parameter set to $\alpha=1.05$ times the local wave speed.
We set $M=50$ uniformly in all TVB tests, following the default choice in \cite{rkdg5}. The CFL numbers are set to $0.05$ for the accuracy tests and $0.15$ for all shock-capturing tests.
The orders of spatial accuracy are chosen to be consistent with the corresponding temporal orders.

\subsection{Scalar conservation laws}\label{sec:scalar_tests}

\subsubsection{Linear advection}\label{sec:A1}
We consider the linear advection equation $u_t + u_x = 0$ on the periodic domain $[0,2\pi]$.

\paragraph{Accuracy test}
We take the smooth initial data
$u(x,0) = \sin(x)$, so that the exact solution is $u(x,t) = \sin(x-t)$.
The final time is $T=1$.
Table~\ref{tab:linear_acc_combined_L2}
reports the $L^2$ errors and convergence orders for the schemes under the TVD ($M=0$) and TVB ($M=50$) settings, respectively.
From the table, we observe optimal accuracy with TVB limiting and order degeneration in the TVD setting for all schemes.

\begin{table}[h!]
\centering\small
\caption{Accuracy tests for 1D linear advection.}\label{tab:linear_acc_combined_L2}
\begin{tabular}{r rr rr rr rr}
\toprule
  $N$ & \multicolumn{2}{c}{$\mathbb{P}^2$ cRK3} & \multicolumn{2}{c}{$\mathbb{P}^3$ cRK4} & \multicolumn{2}{c}{$\mathbb{P}^2$ SSP-RK3} & \multicolumn{2}{c}{$\mathbb{P}^3$ SSP-RK4} \\
\cmidrule(lr){2-9}
  & $L^2$-Error & Order & $L^2$-Error & Order & $L^2$-Error & Order & $L^2$-Error & Order \\
\midrule
  \multicolumn{9}{c}{TVD ($M=0$)} \\
\midrule
     20 & 8.44e-02 & ---  & 4.05e-02 & ---  & 9.06e-02 & ---  & 4.09e-02 & ---  \\
     40 & 2.40e-02 & 1.82 & 1.07e-02 & 1.92 & 2.60e-02 & 1.80 & 1.08e-02 & 1.92 \\
     80 & 6.74e-03 & 1.83 & 2.81e-03 & 1.93 & 7.44e-03 & 1.80 & 2.80e-03 & 1.94 \\
    160 & 1.78e-03 & 1.92 & 7.22e-04 & 1.96 & 1.92e-03 & 1.96 & 7.00e-04 & 2.00 \\
\midrule
  \multicolumn{9}{c}{TVB ($M=50$)} \\
\midrule
     20 & 2.69e-04 & ---  & 5.00e-06 & ---  & 3.12e-04 & ---  & 5.98e-06 & ---  \\
     40 & 3.37e-05 & 3.00 & 3.08e-07 & 4.02 & 3.91e-05 & 3.00 & 3.85e-07 & 3.96 \\
     80 & 4.21e-06 & 3.00 & 1.92e-08 & 4.00 & 4.89e-06 & 3.00 & 2.38e-08 & 4.02 \\
    160 & 5.27e-07 & 3.00 & 1.20e-09 & 4.00 & 6.11e-07 & 3.00 & 1.49e-09 & 3.99 \\
\bottomrule
\end{tabular}
\end{table}
\paragraph{Shock-capturing test}
We solve the same equation with the discontinuous square-wave initial condition $u(x,0) = 1$ for $\pi/2 \le x \le 3\pi/2$, and $u(x,0) = 0$ otherwise, using $N=80$ cells. The final time is
$T=2\pi$, so the exact solution travels one period.
Figures~\ref{fig:linear_disc_tvd} and~\ref{fig:linear_disc_tvb} compare
the cell averages of $\mathbb{P}^2$ cRK3 and $\mathbb{P}^2$ SSP-RK3 solutions against the exact solution
under the TVD and TVB limiters, respectively.
From the figures we can observe almost identical TVDM behaviors of both schemes, and comparable results in the TVB setting. 
 
\begin{figure}[h!]
\centering
\begin{subfigure}[b]{0.45\textwidth}
  \centering
  \includegraphics[width=\textwidth]{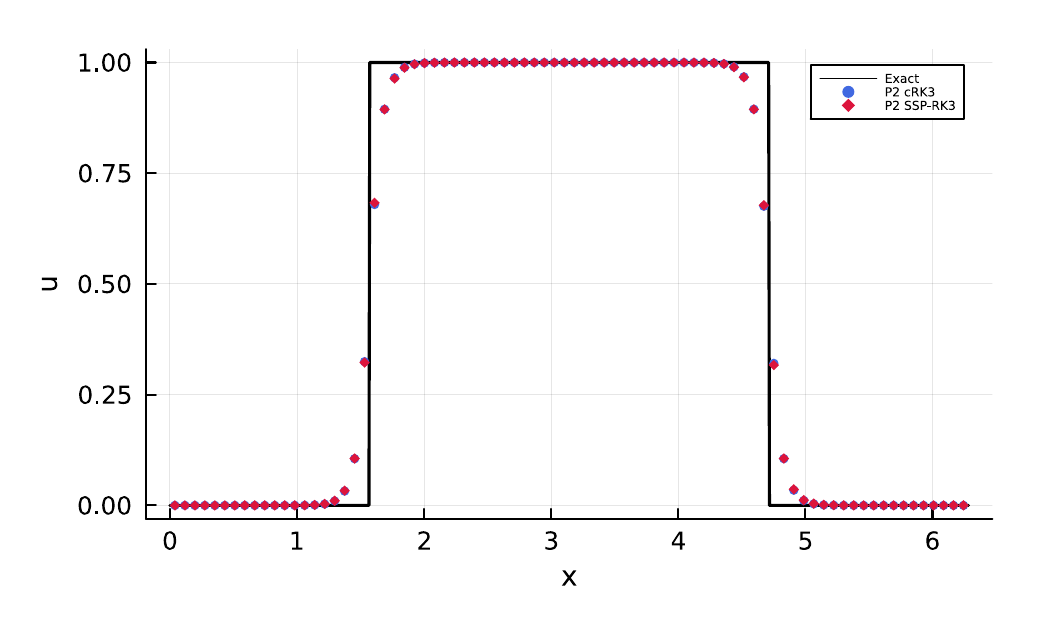}
  \caption{TVD ($M=0$).}
  \label{fig:linear_disc_tvd}
\end{subfigure}
\hfill
\begin{subfigure}[b]{0.45\textwidth}
  \centering
  \includegraphics[width=\textwidth]{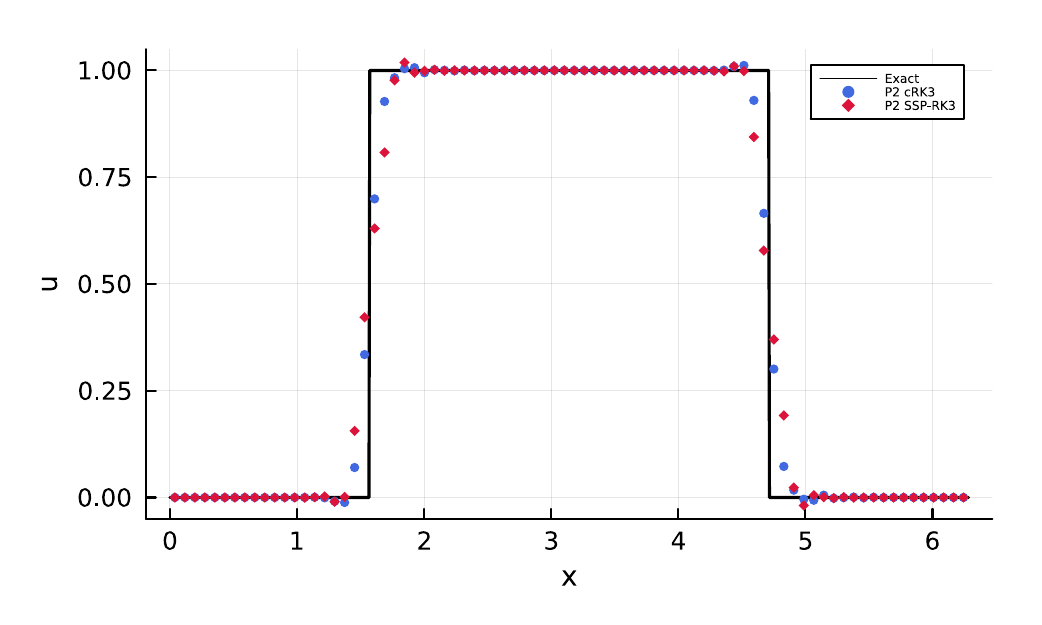}
  \caption{TVB ($M=50$).}
  \label{fig:linear_disc_tvb}
\end{subfigure}
\caption{Linear advection with square-wave initial data.
  $\mathbb{P}^2$ DG, $N=80$, $T=2\pi$.}
\label{fig:linear_disc}
\end{figure}

\subsubsection{Burgers equation}\label{sec:A2}

We consider the Burgers equation 
$u_t + ({u^2}/{2})_x = 0$
on the periodic domain $[0,2\pi]$.

\paragraph{Accuracy test}
We take the smooth initial data
$u(x,0) = {1}/{2}+\sin(x)$.
The exact solution is obtained by the method of characteristics and
Newton iteration.
The final time $T=0.2$ is before shock formation, so the
solution remains smooth throughout the computation.
Table~\ref{tab:burgers_acc_combined_L2}
reports the $L^2$ errors and convergence orders. 
As in the linear case, we observe optimal convergence orders under the TVB limiting and order degeneration under the TVD limiting.

\begin{table}[h!]
\centering\small
\caption{Accuracy tests for 1D Burgers equation.}\label{tab:burgers_acc_combined_L2}
\begin{tabular}{r rr rr rr rr}
\toprule
  $N$ & \multicolumn{2}{c}{$\mathbb{P}^2$ cRK3} & \multicolumn{2}{c}{$\mathbb{P}^3$ cRK4} & \multicolumn{2}{c}{$\mathbb{P}^2$ SSP-RK3} & \multicolumn{2}{c}{$\mathbb{P}^3$ SSP-RK4} \\
\cmidrule(lr){2-9}
  & $L^2$-Error & Order & $L^2$-Error & Order & $L^2$-Error & Order & $L^2$-Error & Order \\
\midrule
  \multicolumn{9}{c}{TVD ($M=0$)} \\
\midrule
     20 & 3.10e-02 & ---  & 1.85e-02 & ---  & 3.22e-02 & ---  & 1.86e-02 & ---  \\
     40 & 8.53e-03 & 1.86 & 4.47e-03 & 2.05 & 9.14e-03 & 1.82 & 4.51e-03 & 2.04 \\
     80 & 2.52e-03 & 1.76 & 1.24e-03 & 1.85 & 2.69e-03 & 1.76 & 1.26e-03 & 1.84 \\
    160 & 6.82e-04 & 1.88 & 3.11e-04 & 1.99 & 7.35e-04 & 1.87 & 3.14e-04 & 2.00 \\
\midrule
  \multicolumn{9}{c}{TVB ($M=50$)} \\
\midrule
     20 & 3.36e-04 & ---  & 1.09e-05 & ---  & 3.55e-04 & ---  & 1.18e-05 & ---  \\
     40 & 4.47e-05 & 2.91 & 7.17e-07 & 3.92 & 4.73e-05 & 2.91 & 7.82e-07 & 3.92 \\
     80 & 5.82e-06 & 2.94 & 4.62e-08 & 3.96 & 6.16e-06 & 2.94 & 5.04e-08 & 3.96 \\
    160 & 7.46e-07 & 2.96 & 2.97e-09 & 3.96 & 7.90e-07 & 2.96 & 3.25e-09 & 3.95 \\
\bottomrule
\end{tabular}
\end{table}
\paragraph{Shock-capturing test}
We solve the problem with the discontinuous square-wave initial condition $u(x,0) = 1$ for $\pi/2 \le x \le 3\pi/2$, and $u(x,0) = -1$ otherwise, using $N=80$ cells and a final time of $T=1$.
At $x=\pi/2$ the jump from $-1$ to $1$ produces a rarefaction
fan, while at $x=3\pi/2$ the jump from $1$ to $-1$ produces a stationary
shock.
This test simultaneously examines the resolution of a rarefaction wave
and the sharpness of a standing shock.
Figures~\ref{fig:burgers_disc_tvd} and~\ref{fig:burgers_disc_tvb}
compare the cell averages of $\mathbb{P}^2$ cRK3 and $\mathbb{P}^2$ SSP-RK3 solutions under the TVD and TVB limiters, respectively.

\begin{figure}[h!]
\centering
\begin{subfigure}[b]{0.45\textwidth}
  \centering
  \includegraphics[width=\textwidth]{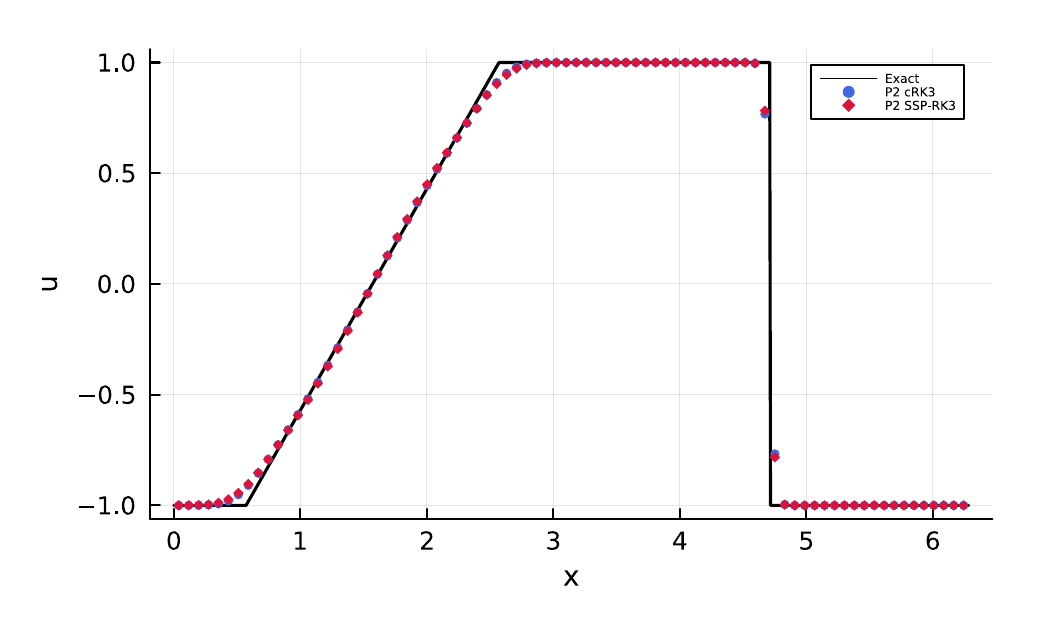}
  \caption{TVD ($M=0$).}
  \label{fig:burgers_disc_tvd}
\end{subfigure}
\hfill
\begin{subfigure}[b]{0.45\textwidth}
  \centering
  \includegraphics[width=\textwidth]{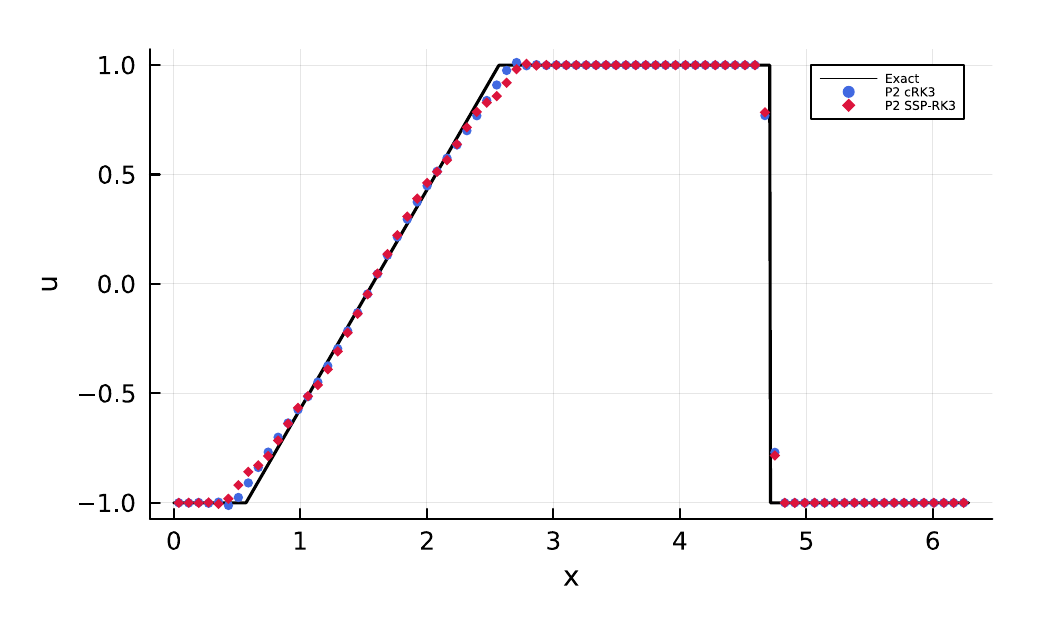}
  \caption{TVB ($M=50$).}
  \label{fig:burgers_disc_tvb}
\end{subfigure}
\caption{Burgers equation with piecewise $\pm 1$ initial data.
  $\mathbb{P}^2$ DG, $N=80$, $T=1$.}
\label{fig:burgers_disc}
\end{figure}

\subsubsection{Buckley--Leverett equation}\label{sec:A3}

We consider the Buckley--Leverett equation $u_t + f(u)_x = 0$, where the flux function is given by $f(u) = 4u^2 / (4u^2 + (1-u)^2)$. 
The flux $f$ is nonconvex in this problem.
We test two Riemann problems.
For both tests, the domain is $[-0.5,\,0.5]$ with $N=80$ cells,
the final time is $T=1$, and inflow boundary conditions are imposed
using the corresponding constant states.

\paragraph{Shock–rarefaction–shock structure}
The initial data are $u(x,0) = 2$ for $x < 0$, and $u(x,0) = -2$ for $x \ge 0$.
The entropy
solution is a composite shock–rarefaction–shock pattern.
Figures~\ref{fig:bl1_tvd} and~\ref{fig:bl1_tvb} show the results.

\begin{figure}[h!]
\centering
\begin{subfigure}[b]{0.45\textwidth}
  \centering
  \includegraphics[width=\textwidth]{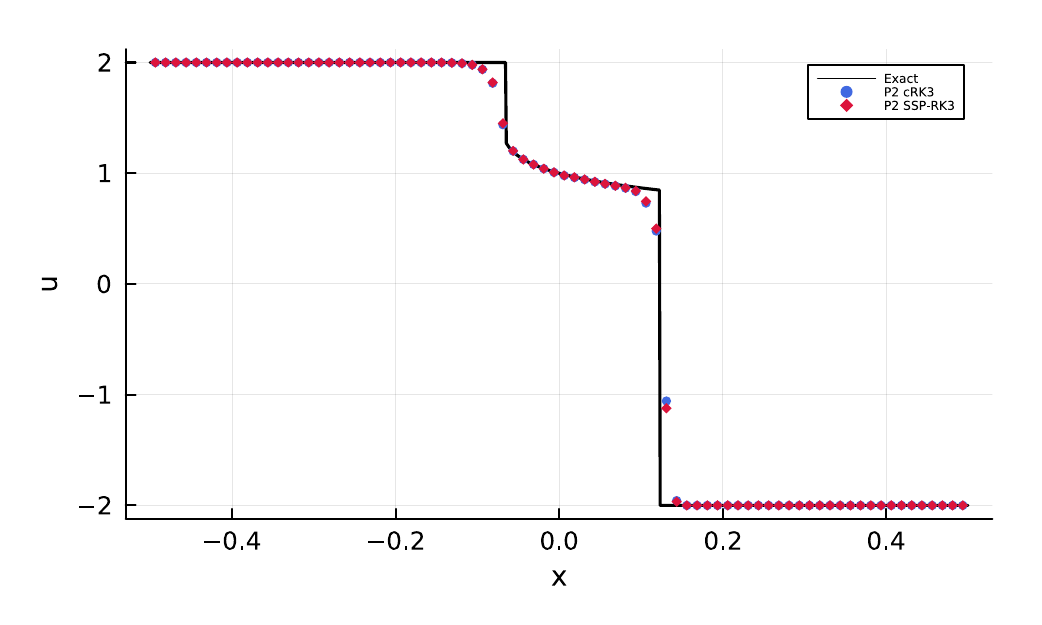}
  \caption{TVD ($M=0$).}
  \label{fig:bl1_tvd}
\end{subfigure}
\hfill
\begin{subfigure}[b]{0.45\textwidth}
  \centering
  \includegraphics[width=\textwidth]{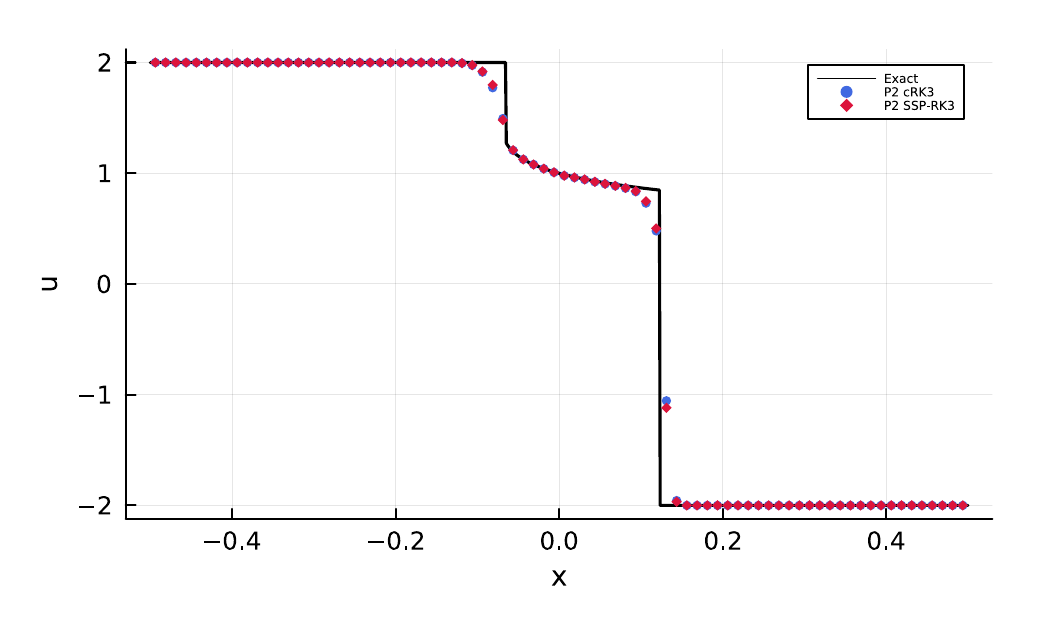}
  \caption{TVB ($M=50$).}
  \label{fig:bl1_tvb}
\end{subfigure}
\caption{Buckley--Leverett, first Riemann problem ($u_L=2$, $u_R=-2$).
  Composite shock--rarefaction--shock structure.
  $\mathbb{P}^2$ DG, $N=80$, $T=1$.}
\label{fig:bl1}
\end{figure}

\paragraph{Two-shock structure}
The initial data are $u(x,0) = -3$ for $x < 0$, and $u(x,0) = 3$ for $x \ge 0$. 
The entropy solution consists of two shocks separated by a constant
intermediate state $u=0$.
Figures~\ref{fig:bl2_tvd} and~\ref{fig:bl2_tvb} show the results.

\begin{figure}[h!]
\centering
\begin{subfigure}[b]{0.45\textwidth}
  \centering
  \includegraphics[width=\textwidth]{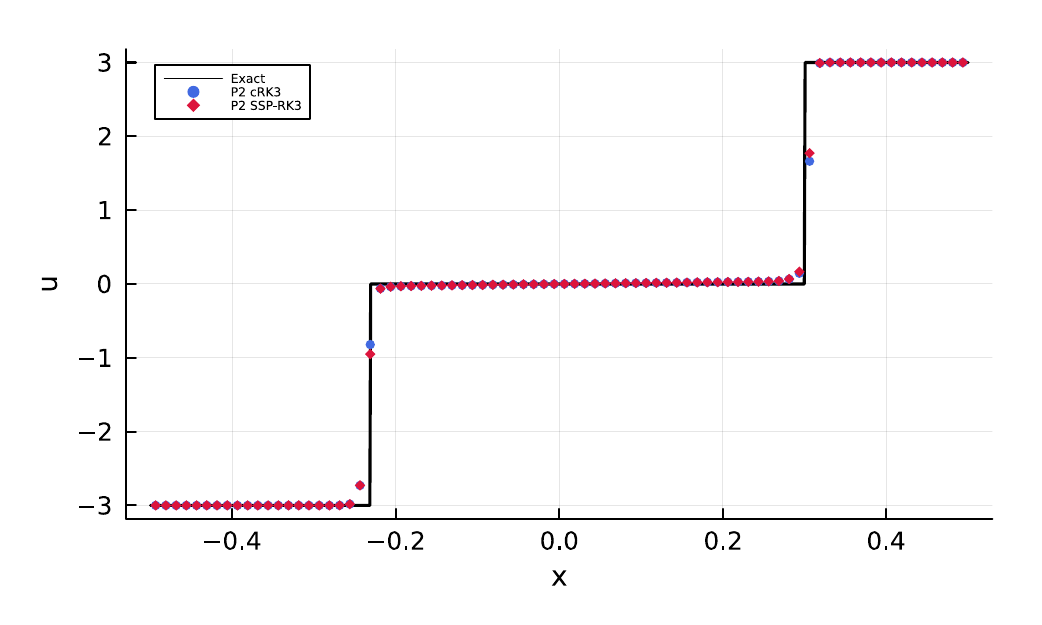}
  \caption{TVD ($M=0$).}
  \label{fig:bl2_tvd}
\end{subfigure}
\hfill
\begin{subfigure}[b]{0.45\textwidth}
  \centering
  \includegraphics[width=\textwidth]{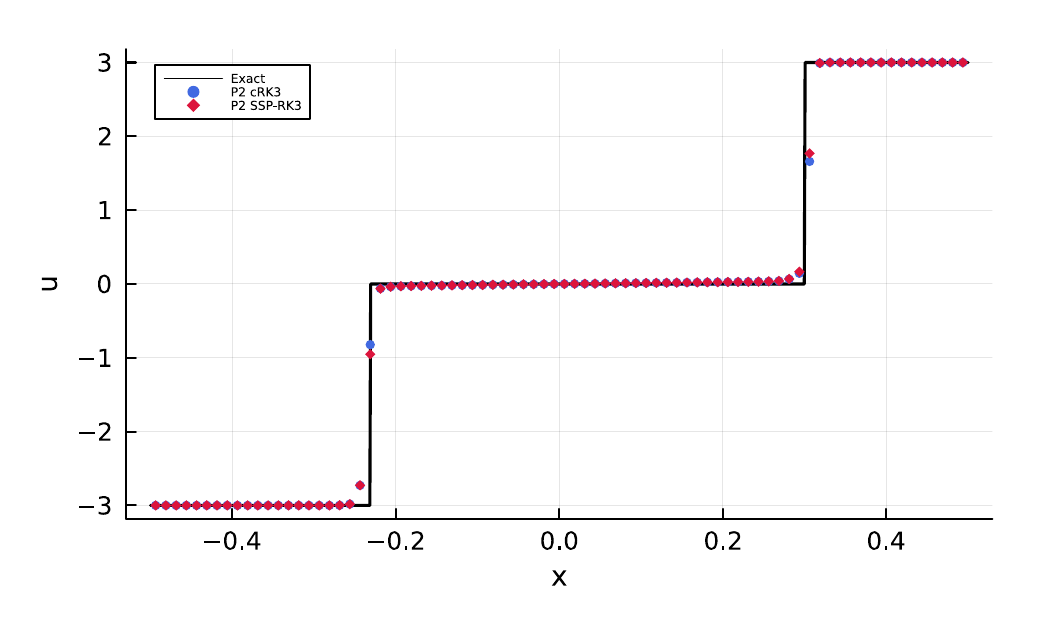}
  \caption{TVB ($M=50$).}
  \label{fig:bl2_tvb}
\end{subfigure}
\caption{Buckley--Leverett, second Riemann problem ($u_L=-3$, $u_R=3$).
  Two-shock structure with intermediate constant state $u=0$.
  $\mathbb{P}^2$ DG, $N=80$, $T=1$.}
\label{fig:bl2}
\end{figure}

\subsubsection{2D Burgers equation}\label{sec:A4}
We consider the 2D Burgers equation $u_t + ({u^2}/{2})_x + ({u^2}/{2})_y = 0$ on the periodic domain $[0,2\pi]^2$.
The initial condition is taken as $u(x,y,0)=1/2+\sin(x+y)$.
The corresponding shock formation time is $T_s=1/2.$

\paragraph{Accuracy test}
The final time $T=0.2$ is chosen before shock formation. Table~\ref{tab:burgers2D_acc_combined_L2} reports the $L^2$ errors and convergence orders. 

\paragraph{Shock-capturing test}
The final time $T=0.7$ is after shock formation. 
Figures \ref{fig:burgers2d_tvd_slice_y} and \ref{fig:burgers2d_tvb_slice_y} show the cell averages of the solutions along the horizontal line $y=0.5$.

\begin{table}[h!]
\centering\small
\caption{Accuracy tests for 2D Burgers equation, $N_x = N_y$.}\label{tab:burgers2D_acc_combined_L2}
\begin{tabular}{r rr rr rr rr}
\toprule
  $N_x$ & \multicolumn{2}{c}{$\mathbb{Q}^2$ cRK3} & \multicolumn{2}{c}{$\mathbb{Q}^3$ cRK4} & \multicolumn{2}{c}{$\mathbb{Q}^2$ SSP-RK3} & \multicolumn{2}{c}{$\mathbb{Q}^3$ SSP-RK4} \\
\cmidrule(lr){2-9}
  & $L^2$-Error & Order & $L^2$-Error & Order & $L^2$-Error & Order & $L^2$-Error & Order \\
\midrule
  \multicolumn{9}{c}{TVD ($M=0$)} \\
\midrule
     20 & 1.05e-01 & ---  & 1.11e-01 & ---  & 1.08e-01 & ---  & 1.15e-01 & ---  \\
     40 & 2.42e-02 & 2.12 & 2.49e-02 & 2.16 & 2.47e-02 & 2.12 & 2.56e-02 & 2.17 \\
     80 & 5.54e-03 & 2.13 & 5.56e-03 & 2.16 & 5.80e-03 & 2.09 & 5.69e-03 & 2.17 \\
    160 & 1.48e-03 & 1.90 & 1.39e-03 & 2.00 & 1.52e-03 & 1.93 & 1.44e-03 & 1.98 \\
\midrule
  \multicolumn{9}{c}{TVB ($M=50$)} \\
\midrule
     20 & 2.58e-03 & ---  & 1.88e-04 & ---  & 2.63e-03 & ---  & 1.94e-04 & ---  \\
     40 & 3.52e-04 & 2.88 & 1.28e-05 & 3.88 & 3.61e-04 & 2.87 & 1.33e-05 & 3.87 \\
     80 & 4.78e-05 & 2.88 & 8.42e-07 & 3.92 & 4.92e-05 & 2.87 & 8.82e-07 & 3.91 \\
    160 & 6.38e-06 & 2.91 & 5.51e-08 & 3.93 & 6.57e-06 & 2.91 & 5.78e-08 & 3.93 \\
\bottomrule
\end{tabular}
\end{table}

\begin{figure}[h!]
\centering
\begin{subfigure}[b]{0.45\textwidth}
  \centering
  \includegraphics[width=\textwidth]{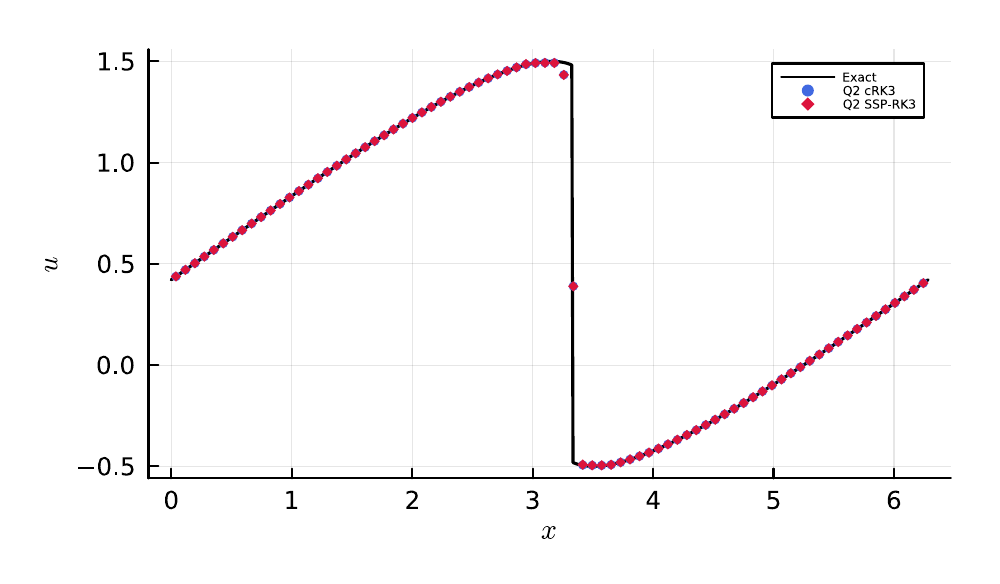}
  \caption{TVD ($M=0$).}
  \label{fig:burgers2d_tvd_slice_y}
\end{subfigure}
\hfill
\begin{subfigure}[b]{0.45\textwidth}
  \centering
  \includegraphics[width=\textwidth]{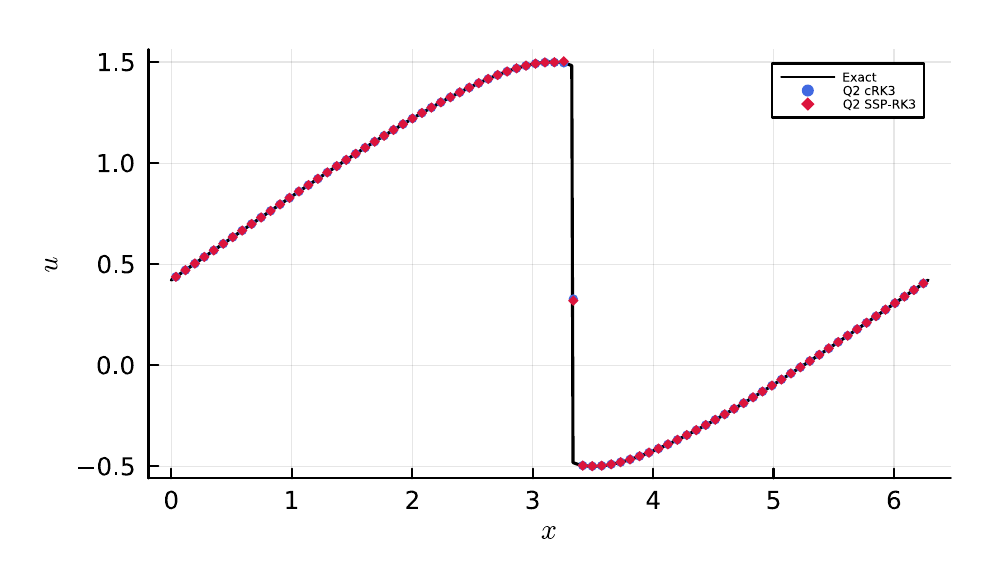}
  \caption{TVB ($M=50$).}
  \label{fig:burgers2d_tvb_slice_y}
\end{subfigure}
\caption{2D Burgers equation with shock, 
  $\mathbb{Q}^2$ DG, $N_x=N_y=80$, $T=0.7$, $y = 0.5$.}
\label{fig:burgers2d}
\end{figure}

\subsection{Euler equations}\label{sec:euler_tests}

We now test the schemes on the 1D Euler equations,
\begin{equation}
    \frac{\partial}{\partial t}
  \begin{pmatrix} \rho \\ m \\ E \end{pmatrix}
  + \frac{\partial}{\partial x}
  \begin{pmatrix} m \\ m^2/\rho + p \\ (E+p)\,m/\rho \end{pmatrix}
  = 0,
\end{equation}
closed by the ideal-gas equation of state
$p = (\gamma-1)\bigl(E - m^2/(2\rho)\bigr)$.
Unless stated otherwise, $\gamma = 1.4$.

\subsubsection{Accuracy test}\label{sec:B1}
We consider $\gamma = 3$ with the isentropic initial data $\rho = 1 + \sin(x)/2$, $u = {1}/{2}+\sin(x)$, and $p = (1 + \sin(x)/2)^3$ on the periodic domain $[0,2\pi]$. 
The system reduces to two
decoupled Burgers equations for the variables
$r_1 = u + \sqrt{3}\,\rho$ and $r_2 = u - \sqrt{3}\,\rho$ before shock formation.
An exact solution is therefore available via the method of characteristics.
The final time is $T=0.1$.
Table~\ref{tab:euler_acc_combined_L2}
reports the $\Ltwo$ errors in energy ($E$)
and the corresponding convergence orders for all schemes under both limiter settings.
\begin{table}[h!]
\centering\small
\caption{Accuracy tests for 1D Euler equations.}\label{tab:euler_acc_combined_L2}
\begin{tabular}{r rr rr rr rr}
\toprule
  $N$ & \multicolumn{2}{c}{$\mathbb{P}^2$ cRK3} & \multicolumn{2}{c}{$\mathbb{P}^3$ cRK4} & \multicolumn{2}{c}{$\mathbb{P}^2$ SSP-RK3} & \multicolumn{2}{c}{$\mathbb{P}^3$ SSP-RK4} \\
\cmidrule(lr){2-9}
  & $L^2$-Error & Order & $L^2$-Error & Order & $L^2$-Error & Order & $L^2$-Error & Order \\
\midrule
  \multicolumn{9}{c}{TVD ($M=0$)} \\
\midrule
     20 & 1.40e-01 & ---  & 8.99e-02 & ---  & 1.50e-01 & ---  & 9.17e-02 & ---  \\
     40 & 4.04e-02 & 1.80 & 2.26e-02 & 2.00 & 4.33e-02 & 1.79 & 2.29e-02 & 2.00 \\
     80 & 1.20e-02 & 1.75 & 5.49e-03 & 2.04 & 1.32e-02 & 1.72 & 5.47e-03 & 2.07 \\
    160 & 3.32e-03 & 1.86 & 1.43e-03 & 1.94 & 3.59e-03 & 1.88 & 1.40e-03 & 1.96 \\
\midrule
  \multicolumn{9}{c}{TVB ($M=50$)} \\
\midrule
     20 & 1.54e-03 & ---  & 7.48e-05 & ---  & 1.73e-03 & ---  & 8.56e-05 & ---  \\
     40 & 1.90e-04 & 3.02 & 4.88e-06 & 3.94 & 2.14e-04 & 3.01 & 5.66e-06 & 3.92 \\
     80 & 2.33e-05 & 3.03 & 3.06e-07 & 4.00 & 2.61e-05 & 3.04 & 3.49e-07 & 4.02 \\
    160 & 2.87e-06 & 3.03 & 1.90e-08 & 4.01 & 3.20e-06 & 3.03 & 2.19e-08 & 3.99 \\
\bottomrule
\end{tabular}
\end{table}

\subsubsection{Blast wave problem}\label{sec:B3}

The blast wave problem of Woodward and Colella \cite{woodward1984numerical} models the interaction
of strong shocks in a closed tube. The initial data are $(\rho,u,p) = (1, 0, 1000)$ for $x < 0.1$, $(\rho,u,p) = (1, 0, 0.01)$ for $0.1 \le x < 0.9$, and $(\rho,u,p) = (1, 0, 100)$ for $x \ge 0.9$, on the domain $[0,1]$ with reflecting boundary conditions on both ends. The test is run with $N=400$ cells until $T=0.038$.
There is no closed-form solution for this problem.
We use a WENO5 finite-difference solution on $10{,}000$ cells as
the reference.
Figures~\ref{fig:blast_tvd} and~\ref{fig:blast_tvb} show the
density profiles.

\begin{figure}[h!]
\centering
\begin{subfigure}[b]{0.45\textwidth}
  \centering
  \includegraphics[width=\textwidth]{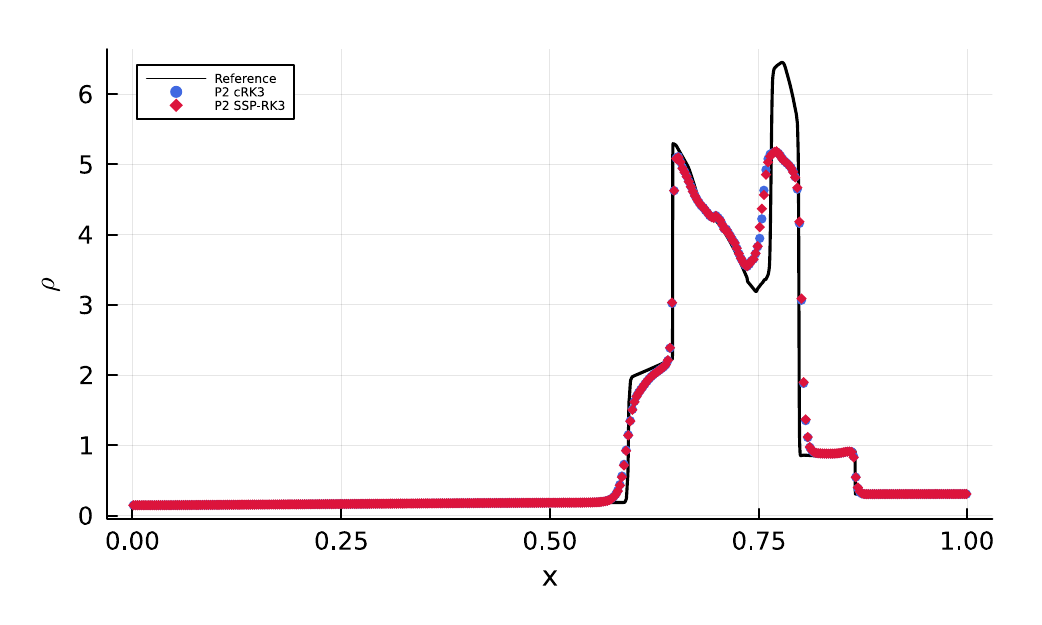}
  \caption{TVD ($M=0$).}
  \label{fig:blast_tvd}
\end{subfigure}
\hfill
\begin{subfigure}[b]{0.45\textwidth}
  \centering
  \includegraphics[width=\textwidth]{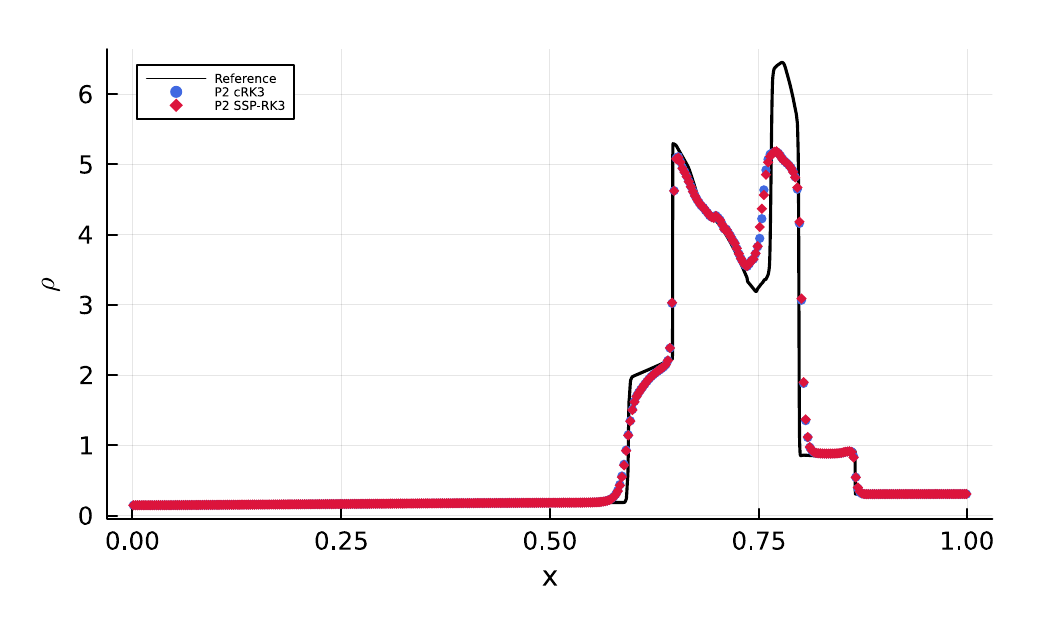}
  \caption{TVB ($M=50$).}
  \label{fig:blast_tvb}
\end{subfigure}
\caption{Woodward--Colella blast wave.
  $\mathbb{P}^2$ DG, $N=400$, $T=0.038$.
  $\rho$ is shown.}
\label{fig:blast}
\end{figure}

\subsubsection{Shu--Osher problem}\label{sec:B4}

The Shu--Osher problem \cite{shu1988efficient} is a standard test used to evaluate if a scheme can
simultaneously handle a strong shock and small-scale smooth structures.
The initial data are $(\rho,u,p) = (3.857143, 2.629369, 10.33333)$ for $x < -4$, and $(\rho,u,p) = (1 + 0.2\sin(5x), 0, 1)$ for $x \ge -4$, on the domain $[-5,5]$ with $N=200$ cells. 
We impose transmissive boundaries and take the final time $T=1.8$.
Since there is no closed-form solution, we use a WENO5 finite-difference solution on $10{,}000$ cells as
the reference.
Figures~\ref{fig:shuosher_tvd} and~\ref{fig:shuosher_tvb} show
the density profiles.

\begin{figure}[h!]
\centering
\begin{subfigure}[b]{0.45\textwidth}
  \centering
  \includegraphics[width=\textwidth]{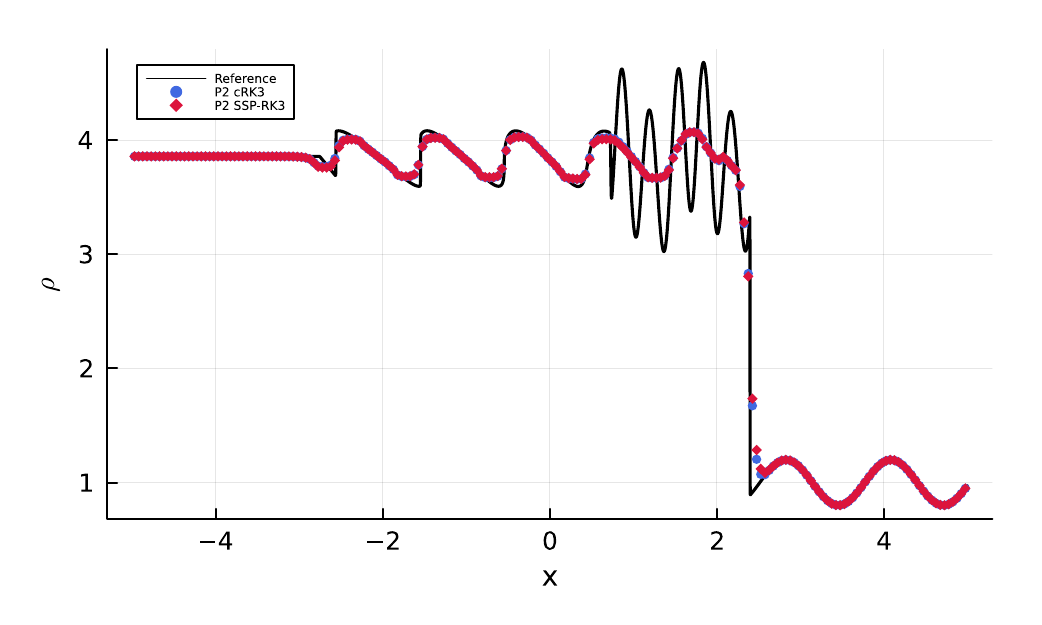}
  \caption{TVD ($M=0$).}
  \label{fig:shuosher_tvd}
\end{subfigure}
\hfill
\begin{subfigure}[b]{0.45\textwidth}
  \centering
  \includegraphics[width=\textwidth]{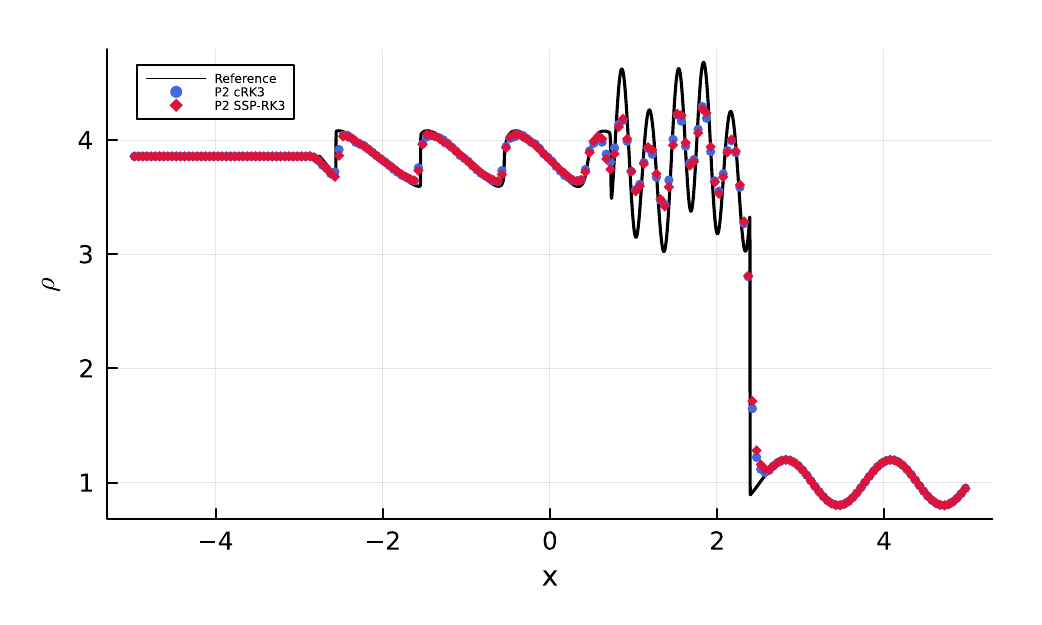}
  \caption{TVB ($M=50$).}
  \label{fig:shuosher_tvb}
\end{subfigure}
\caption{Shu--Osher problem.
  $\mathbb{P}^2$ DG, $N=200$, $T=1.8$.
  $\rho$ is shown.}
\label{fig:shuosher}
\end{figure}

We then test the schemes on the 2D Euler equations
\begin{equation}
    \frac{\partial}{\partial t}
  \begin{pmatrix} \rho \\ \rho u \\ \rho v \\ E \end{pmatrix}
  + \frac{\partial}{\partial x}
  \begin{pmatrix} \rho u \\ \rho u^2 + p \\ \rho u v \\ (E+p)\,u \end{pmatrix}
  + \frac{\partial}{\partial y}
  \begin{pmatrix} \rho v \\ \rho u v \\ \rho v^2 + p \\ (E+p)\,v \end{pmatrix}
  = 0,
\end{equation}
with $p = (\gamma-1)(E - \rho(u^2+v^2)/2)$ and $\gamma = 1.4$.
Only TVB ($M=50$) results are presented to save space.

\subsubsection{Double Mach reflection}\label{sec:D3}

We consider the double Mach reflection problem \cite{woodward1984numerical}.
A Mach~10 oblique shock initially makes a $60^\circ$ angle with the
$x$-axis, meeting the bottom wall at $x=1/6$.
The computational domain is $[0,4]\times[0,1]$.
The post-shock state is
$(\rho,u,v,p) = (8,\,33\sqrt{3}/8,\,-33/8,\,116.5)$
and the pre-shock state is
$(\rho,u,v,p) = (1.4,\,0,\,0,\,1)$.
Along the bottom boundary, the exact post-shock condition is imposed
for $x < 1/6$ and a reflecting wall for $x \ge 1/6$.
Along the top boundary, the exact motion of the oblique shock is
prescribed. 
Inflow and outflow conditions are imposed on the left
and right boundaries, respectively.
The final time is $T=0.2$.
We run the schemes with TVB limiting on two meshes: a coarse mesh of $480\times120$ cells and a fine mesh of
$1920\times480$ cells.
Figure~\ref{fig:dmr} shows that the two schemes produce very close results at both resolutions. 
\begin{figure}[h!]
\centering
\begin{subfigure}[b]{0.45\textwidth}
  \centering
  \includegraphics[width=\textwidth]{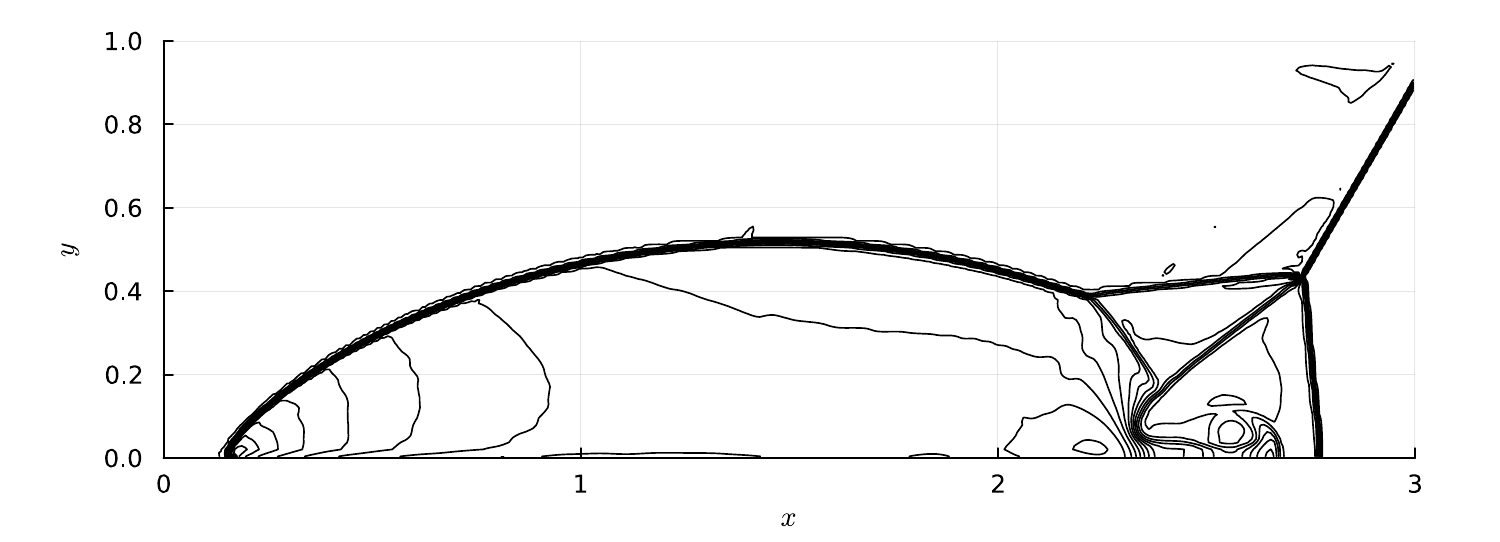}
  \caption{$\mathbb{Q}^2$ cRK3, $480\times120$.}
\end{subfigure}
\hfill
\begin{subfigure}[b]{0.45\textwidth}
  \centering
  \includegraphics[width=\textwidth]{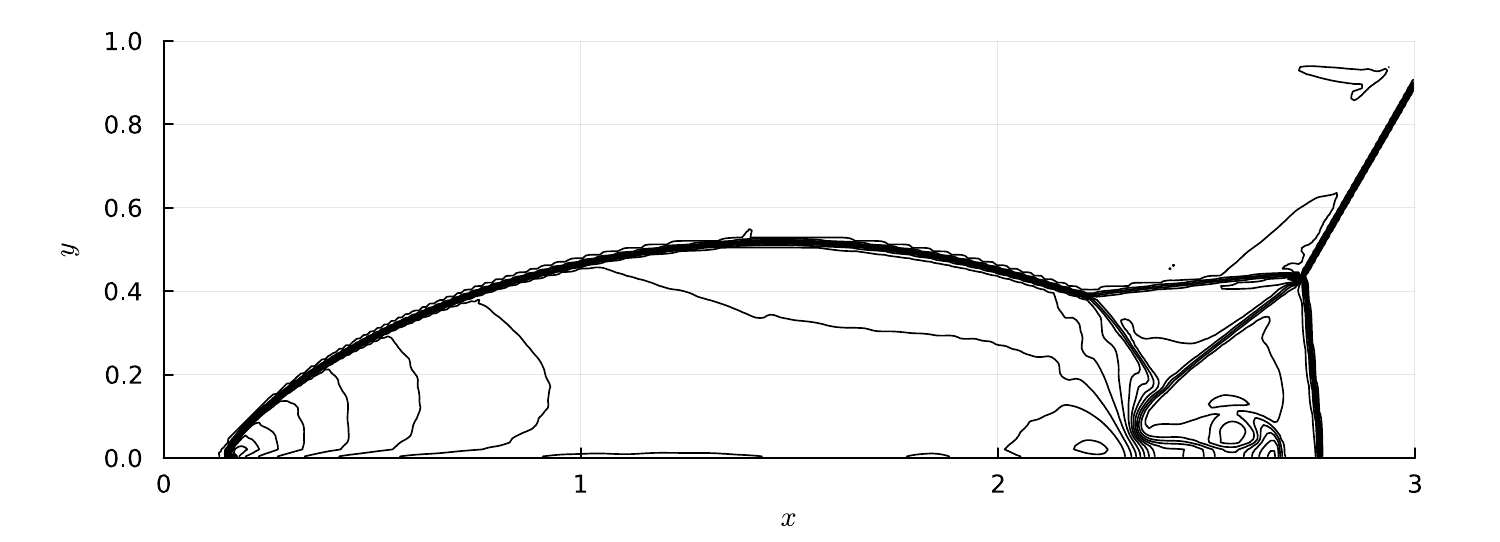}
  \caption{$\mathbb{Q}^2$ SSP-RK3, $480\times120$.}
\end{subfigure}

\medskip

\begin{subfigure}[b]{0.45\textwidth}
  \centering
  \includegraphics[width=\textwidth]{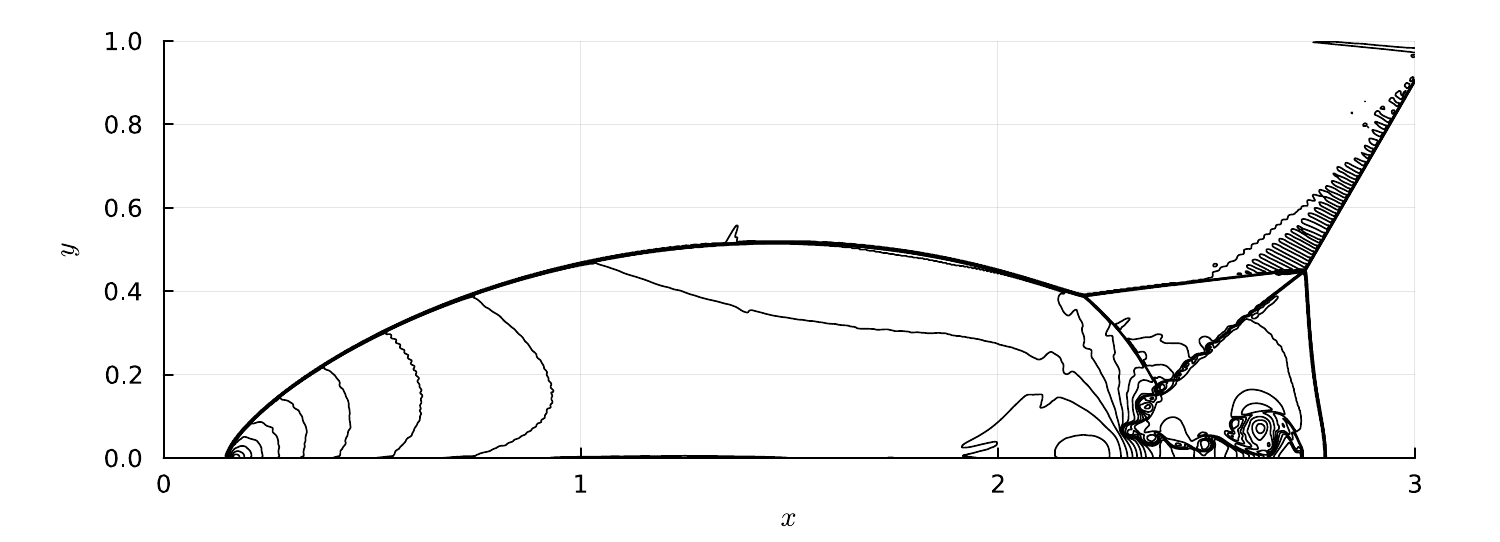}
  \caption{$\mathbb{Q}^2$ cRK3, $1920\times480$.}
\end{subfigure}
\hfill
\begin{subfigure}[b]{0.45\textwidth}
  \centering
  \includegraphics[width=\textwidth]{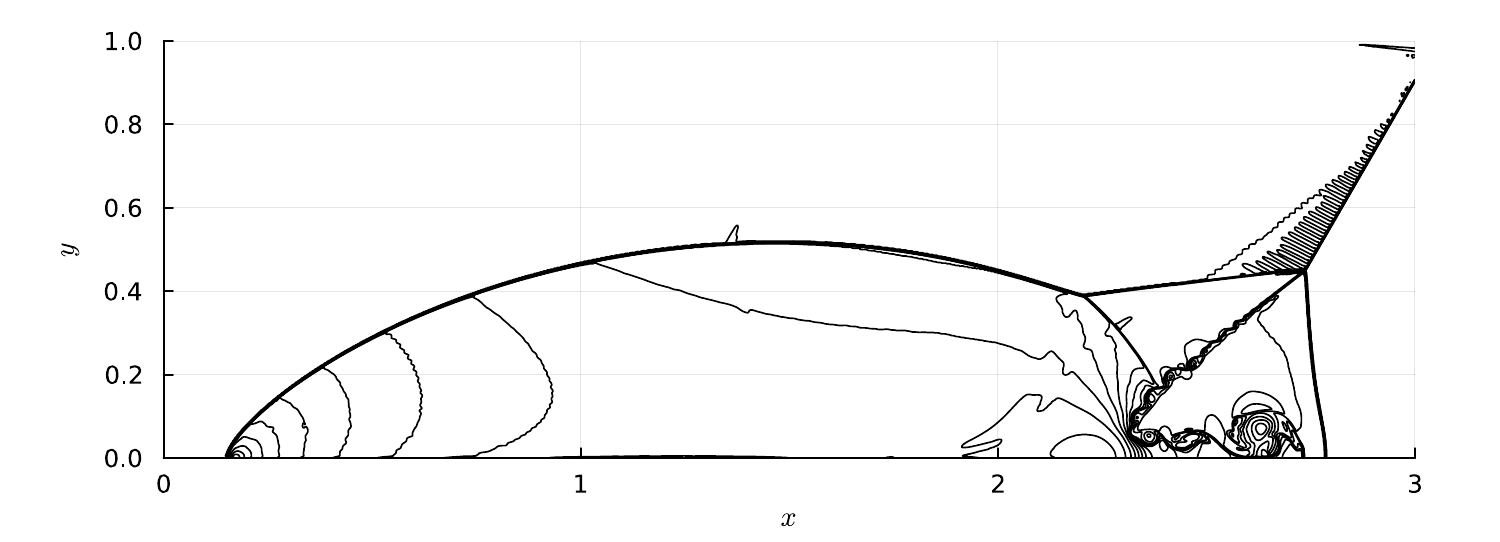}
  \caption{$\mathbb{Q}^2$ SSP-RK3, $1920\times480$.}
\end{subfigure}
\caption{Double Mach reflection, TVB ($M=50$), $T=0.2$.
  30 equally spaced density contours from
$1.3965$ to $22.682$ on the subdomain $[0,3]\times[0,1]$.}
\label{fig:dmr}
\end{figure}

\subsubsection{Forward-facing step}\label{sec:D4}

We consider the forward-facing step problem~\cite{woodward1984numerical}.
A uniform Mach~3 flow enters an L-shaped tunnel with a step
of height $0.2$ located at $x=0.6$.
The computational domain is $[0,3]\times[0,1]$ with the rectangular
region $[0.6,3]\times[0,0.2]$ removed.
The initial condition is $(\rho,u,v,p) = (1.4,\,3,\,0,\,1)$.
Inflow conditions are imposed on the left boundary, outflow on the
right, and reflecting walls on all remaining boundaries.
The final time is $T=4$.
We run the schemes with TVB limiting on two meshes: a coarse mesh of $240\times80$ cells and a fine mesh of
$960\times320$ cells.
Again, the cRK3 and SSP-RK3 results are nearly identical; see Figure~\ref{fig:ffs}.
\begin{figure}[h!]
\centering
\begin{subfigure}[b]{0.45\textwidth}
  \centering
  \includegraphics[width=\textwidth]{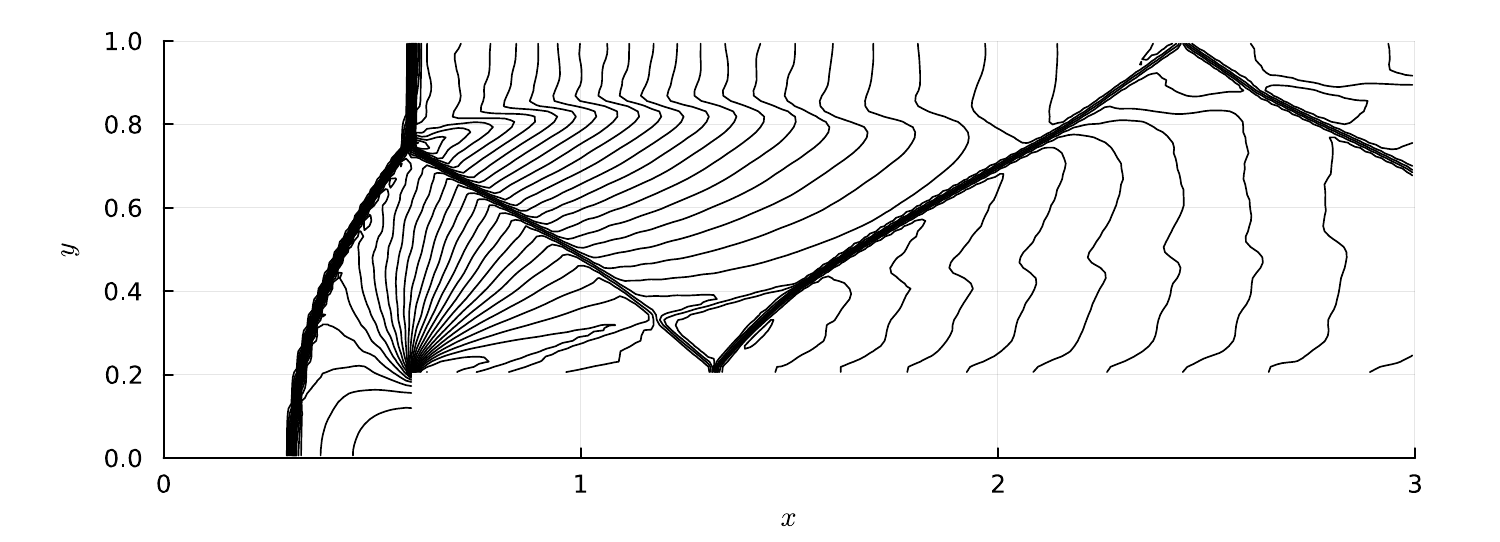}
  \caption{$\mathbb{Q}^2$ cRK3, $240\times80$.}
\end{subfigure}
\hfill
\begin{subfigure}[b]{0.45\textwidth}
  \centering
  \includegraphics[width=\textwidth]{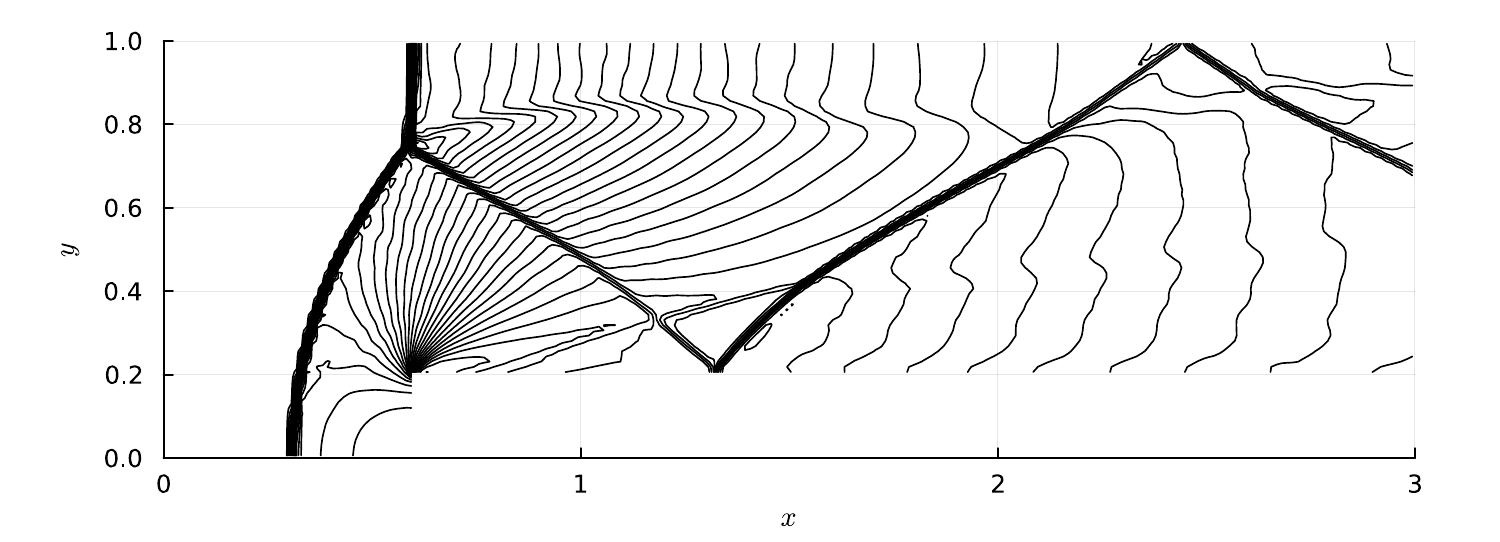}
  \caption{$\mathbb{Q}^2$ SSP-RK3, $240\times80$.}
\end{subfigure}

\medskip

\begin{subfigure}[b]{0.45\textwidth}
  \centering
  \includegraphics[width=\textwidth]{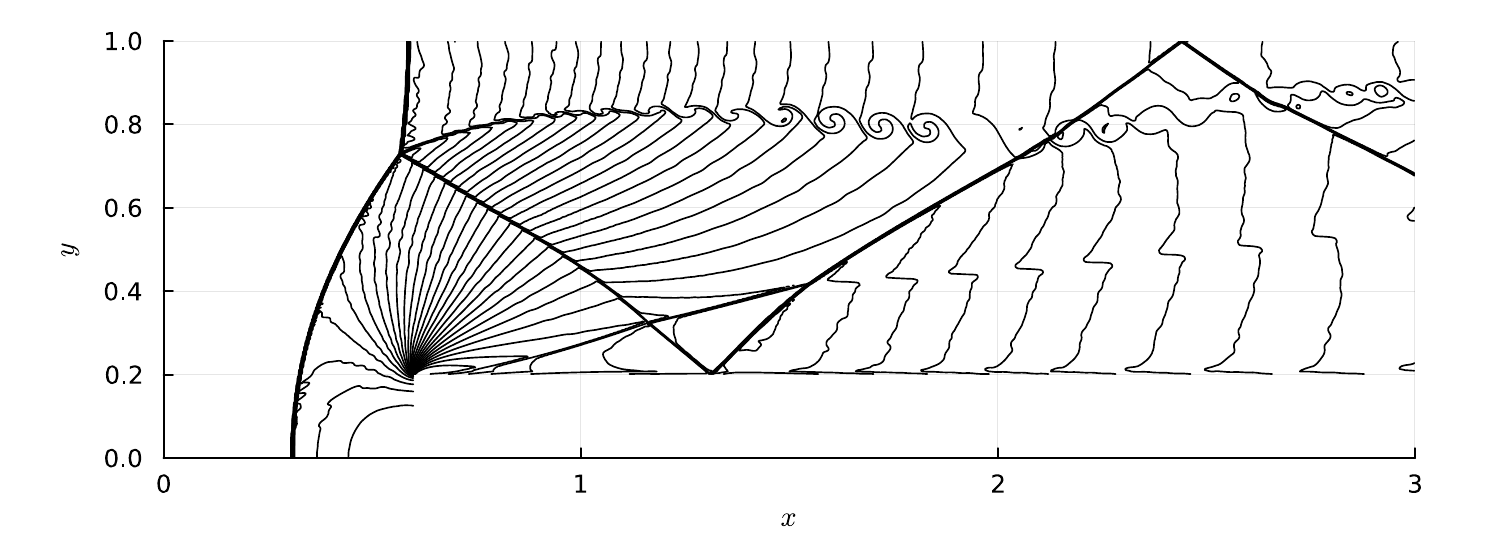}
  \caption{$\mathbb{Q}^2$ cRK3, $960\times320$.}
\end{subfigure}
\hfill
\begin{subfigure}[b]{0.45\textwidth}
  \centering
  \includegraphics[width=\textwidth]{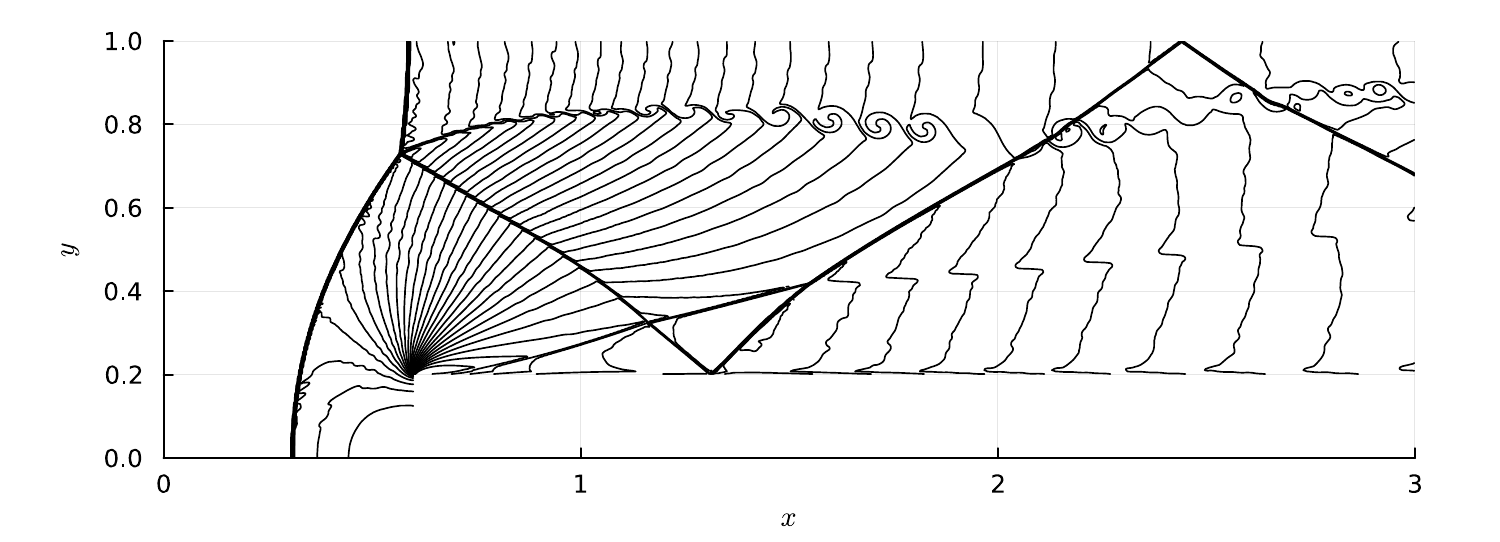}
  \caption{$\mathbb{Q}^2$ SSP-RK3, $960\times320$.}
\end{subfigure}
\caption{Forward-facing step, TVB ($M=50$), $T=4$. 30 equally spaced density contours from
$0.090338$ to $6.2365$.}
\label{fig:ffs}
\end{figure}

\section{Conclusions}\label{sec:conclusions}
{This paper presents a novel framework for preserving the TVD and TVB properties of the DG methods using non-SSP time discretizations. The new approach uses a trace-limiting strategy and limits only the final corrector stage, regardless of how intermediate predictor stages are computed. Based on Harten's lemma, it can be shown that the trace limiting alone can preserve the TVDM/TVBM property of the method. For further robustness, a polynomial limiter is applied at the end of the final stage to produce non-oscillatory DG polynomials. The conservation, TVDM/TVBM, and accuracy-preserving properties are proved. The proposed method decouples the TVD/TVB limiting from the SSP restriction, streamlines the algorithmic flow for concurrent computing, retains the compactness of the cRKDG framework, and accommodates fully discrete schemes beyond the RK family. The algorithmic extension to multidimensional systems is also provided. Numerical experiments with the TVD/TVB cRKDG schemes demonstrate that this new approach achieves performance comparable to classical SSP-RKDG methods. Our future work includes detailed studies on extending the framework to ADER-DG, Lax--Wendroff DG, and implicit RKDG methods, as well as its compatibility with other structure-preserving techniques.}

\section*{Acknowledgment} The authors thank Professor Chi-Wang Shu of Brown University for his helpful comments and suggestions, which greatly improved the paper. The authors report the usage of AI tools and assume responsibility for all content.

\bibliographystyle{abbrv}
\bibliography{refs}

\end{document}

%% file: ex_shared.tex
\usepackage{lipsum}
\usepackage{amsfonts}
\usepackage{graphicx}
\usepackage{epstopdf}
\usepackage{algorithmic}
\ifpdf
  \DeclareGraphicsExtensions{.eps,.pdf,.png,.jpg}
\else
  \DeclareGraphicsExtensions{.eps}
\fi

\newsiamremark{remark}{Remark}
\newsiamremark{hypothesis}{Hypothesis}
\crefname{hypothesis}{Hypothesis}{Hypotheses}
\newsiamthm{claim}{Claim}

\headers{TVD Preservation without TVD Time Discretization}{}

\title{TVD and TVB Preservation in RKDG and cRKDG Methods without TVD Time Discretization}

\author{Ziyao Xu\thanks{Department of Mathematics and Statistics, Binghamton University, Binghamton, NY 13902, USA. (\email{zxu24@binghamton.edu}) }
  \and Zheng Sun\thanks{Department of Mathematics, The University of Alabama,
		Tuscaloosa, AL 35487, USA. (\email{zsun30@ua.edu}) The work of this author was partially supported by the NSF grant DMS-2208391.}}

\usepackage{amsopn}
